\documentclass[11pt]{amsart}
\usepackage{amssymb}
\usepackage{amsfonts}
\usepackage{color}
\usepackage{tikz}
\usepackage{xy}
\xyoption{all}
\usepackage{tablefootnote}
\usepackage{setspace}
\newtheorem{theorem}{Theorem}[section]
\newtheorem{proposition}[theorem]{Proposition}
\newtheorem{lemma}[theorem]{Lemma}
\newtheorem{corollary}[theorem]{Corollary}
\newtheorem{definition}[theorem]{Definition}

\newtheorem{problem}[theorem]{Problem}

\theoremstyle{remark}
\newtheorem{remark}[theorem]{Remark}

\newtheorem{example}[theorem]{\bf Example}
\renewenvironment{proof}{{\noindent\bf Proof.}}{\hfill $\Box$\par\vskip3mm}

\newcommand{\Hom}{{\rm Hom}}

\newcommand{\Ext}{{\rm Ext}}

\newcommand{\Aut}{{\rm Aut}\,}
\newcommand{\Top}{{\rm Top}\,}
\newcommand{\Rad}{{\rm Rad}\,}
\newcommand{\Soc}{{\rm Soc}\,}
\newcommand{\ad}{{\rm ad}\,}

\newcommand{\Cc}{\mathcal{C}}
\newcommand{\Dd}{\mathcal{D}}

\newcommand{\Mm}{\mathcal{M}}

\newcommand{\modd}{\textendash\operatorname{mod}}
\newcommand{\cmodd}{\textendash\operatorname{comod}}

\newcommand{\Rep}{\operatorname{Rep}}

\newcommand{\Span}{\operatorname{Span}}
\newcommand{\Set}{\mathbf{Set}}

\def\KK{{\mathbb K}}

\begin{document}
\title[Pointed Hopf Algebras of discrete corepresentation type]{Pointed Hopf Algebras of discrete corepresentation type}
%\title[Serial Representations of Infinite Dimensional Algebras]{Serial Representations of Infinite Dimensional Algebras}

%{Serial Representations, Quantum Groups, and the Pure Semisimplicity Problem}
%OLD [Serial Categories, Quantum Groups and Pure Semisimplicity]{Serial Categories, Quantum Groups and the Pure Semisimplicity Problem}%Serial representations of algebras and applications
%{Serial Linear Categories, Infinite Abelian Groups, Quantum Groups and Open Questions in Corepresentation Theory}
%\dedicatory{}

\begin{abstract}
We classify pointed Hopf algebras of discrete corepresentation type over an algebraically closed field $\KK$ with characteristic zero. For such algebras $H$, we explicitly determine the algebra structure up to isomorphism for the link indecomposable component $B$ containing the unit. It turns out that $H$ is a crossed product of $B$ and a certain group algebra.
\end{abstract}

\author{Miodrag Cristian Iovanov, Emre Sen, Alexander Sistko, Shijie Zhu\\}%$^*$}
% 16G60, 16T05}%, 18E40, 15A21,  05C25}%, 20N99, 05E10}
%Primary 16W30; Secondary 16S90, 16Lxx, 16Nxx, 18E40}
%\thanks{$^*$}
\date{}
\maketitle

%\pagestyle{myheadings}

%\section*{Introduction}

% to set color for comment numbers, use any of the 68 colors from
% the color package in the command below
\newcommand{\margincolor}{red}

% control the width of your comments
\addtolength{\marginparwidth}{-10mm}

\newcounter{margincounter}
\setcounter{margincounter}{1}

\newcommand{\marginnum}{\textcolor{\margincolor}{\begin{picture}(0,0)\put(5,3){\circle{13}}\end{picture}\arabic{margincounter}}}

\newcommand{\margin}[1]
{\marginnum\marginpar{\textcolor{\margincolor}{\arabic{margincounter}.}\,\,\tiny #1}\addtocounter{margincounter}{1}}
%  to remove marginal notes, uncomment the following:
% \renewcommand{\margin}[1]{}
%  to remove just the circled numbers in the text, uncomment the following:
 % \renewcommand{\marginnum}{}
%  For final versions of a paper, it's probably best to remove all the \margin
%  commands.  Much to my annoyance, they mess up the typesetting.

\setcounter{tocdepth}{1} %show only sections in the content
\tableofcontents

\section{Introduction}% and Preliminaries}
%\noindent
\subsection{Background}
 In this paper, we study pointed Hopf algebras and the representation types of the comodule categories. In representation theory, due to Gabriel, each basic algebra can be constructed by a quiver with some relations.  Dually, pointed coalgebras can be realized as subcoalgebras of a path coalgebra of some quiver \cite{CM, DNR}. This allows us to construct pointed Hopf algebras using a basis of paths of some quiver. Although for a finite dimensional coalgebra $C$, the category of right comodules over $C$ is equivalent to the category of left modules over its dual algebra $C^*$, it is more natural to consider comodules in the study of quantum groups or algebraic groups, as they correspond to the rational representations. For this reason, we focus on analyzing the comodule categories and their corepresentation types.  We are going to remind the readers some basic facts about Hopf algebras and representations in $\S\ref{Sec:preliminaries}$.

%[TBD: what we say about what Hopf algebras are; might say something in the introduction about their comodule categories being regarded as representations of some abstract quantum group]
%\margin{TBD}

One of the most important topics in representation theory is the classification of indecomposable (co)modules over a (co)algebra. A  (co)algebra is finite representation type if there are only finitely many non-isomorphic indecomposable (co)modules over the (co)algebra. The (co)module category of a finite (co)representation type (co)algebra is considered easiest to understand. 
The study of representation finite type algebras dates back to the classification of irreducible representations of finite groups and Gabriel's classification of finite representation type path algebras.
 In \cite{LL}, the authors classified finite dimensional basic (pointed) Hopf algebras over an algebraically closed field $\KK$ which are finite (co)representation type. Together with an earlier result of the classification of monomial Hopf algebras \cite{ChHYZ}, one can conclude that a finite dimensional basic Hopf algebra over an algebraically closed field $\KK$ is finite representation type if and only if it is monomial.  
 
However, concerning infinite dimensional Hopf algebras, it is no longer appropriate to discuss the (co)representation finiteness. Instead, we shall consider (co)representation discrete type (co)algebras as analogs of finite dimensional (co)representation finite type (co)algebras.  Following Simson, we formally introduce this definition:
 
\begin{definition}\label{defn:discrete type}
Let $C$ be a pointed coalgebra. We say that $C$ is of discrete corepresentation type, if for any finite dimension vector $\underline{d}$, there are only finitely many indecomposable $C$-comodules of dimension vector $\underline{d}$.
\end{definition}

Our main task in this paper is to classify pointed Hopf algebras of discrete corepresentation type.
We note immediately that according to the validity of the Brauer-Thrall problem over algebraically closed fields, if $C$ is a finite dimensional coalgebra, $C$ is of discrete corepresentation type if and only if $C^*$ is an algebra of finite representation type. Any coalgebra $C$ is of discrete corepresentation type if and only if every finite dimensional subcoalgebra $D$ of $C$ is of finite corepresentation type, which is to say that $D^*$ is of finite representation type (See Proposition \ref{discrete type dual}).  Hence discrete corepresentation type can be regarded as a situation which is ``locally" finite corepresentation type. This can, in fact, be rigorously interpreted and formulated in terms of non-commutative localization \cite{IP}. 
%We will also say, in this case, that $C$ is of finite representation type. \margin{What is the case? discrete=finite?} Of course, if $C$ has infinitely many simple comodules (its quiver has infinitely many vertices), it is impossible to be of finite representation type in the usual meaning of the term.
 
Notice that one class of corepresentation discrete type pointed Hopf algebras has been classified in \cite {IJ}.
\begin{definition}\label{coserial defn}
A left (right) comodule $M$ is called {\bf uniserial} if its lattice of subcomodules is a chain. A coalgebra is {\bf left (right) serial} if it is a direct sum of left (right) uniserial comodules. A coalgebra is called {\bf serial} if it is left and right serial. A Hopf algebra $H$ is called {\bf coserial} if it is serial as a coalgebra.
\end{definition}
It turns out that  coserial  pointed Hopf algebras are of corepresentation discrete type. We show there are another classes of corepresentation discrete type pointed Hopf algebras other than the coserial ones (see Theorem \ref{ext quiver} below.).

\subsection{Outline of the paper}
In section 2, we recall some basic facts about coalgebras, comodules and $\Ext$-quivers of  comodule categories. 

In section 3, we introduce the notion of covering maps between coalgebras (Definition \ref{covering def}), which comes from separable extensions of algebras. It serves as a main tool determining corepresentation finite/discrete type due to the Lemma below:

 \begin{lemma} 
If $\pi:C\to D$ is a covering map, then $\pi^*: D^*\to C^*$ is a separable extension of algebras. Hence if  $D$ is corepresentation finite, then so is $C$.
\end{lemma}

In section 4, assuming $H$ is a 
pointed Hopf algebra of discrete corepresentation type with $B$ its link-indecomposable component containing $1$, we classify $\Ext$-quivers of the comodule category $\Mm^B$ in {\bf Theorem \ref{ext quiver}}. It turns out that the $\Ext$-quiver is either of types 
\begin{enumerate}
\item a single point;
 \item an oriented cycle $\widetilde A_n$; 
 \item double infinite quiver $_\infty A_\infty$;  which corresponds to $B$ being a coserial Hopf algebra;
 \item a certain quiver which satisfies for each vertex there are exactly two arrows coming in and going out. 
\end{enumerate}
Cases (1)-(3) coincides with the case  when $H$ is a coserial Hopf algebra, which has been classified in \cite{IJ}.

We deal with the remaining case in section 5.  In {\bf Theorem \ref{B is discrete}}, we determine the algebra structures of link indecomposable discrete corepresentation type Hopf algebras. Up to isomorphism, such a Hopf algebra is isomorphic to one of the algebra $B^{m,n}(\lambda,s,t, k)$ in the list of {\bf Theorem \ref{classify B}}. For each isomorphism class, the automorphism group of the Hopf algebra is explicitly computed as well.

In section 6,  we  provide a full list of  pointed Hopf algebras of discrete corepresentation type over an algebraically closed field $\KK$ with $char \KK=0$. They are either coserial or isomorphic to the smash product of a Hopf algebra $B^{m,n}(\lambda,s,t, k)$ in the list of Theorem \ref{classify B} with a group algebra (see {\bf Theorem \ref{H main}}).

%For a pointed Hopf algebra $H$, there is a link indecomposable component $B$ of $H$ containing $1$. This is actually a Hopf subalgebra of $H$. It turns out that $H$ is obtained by ``translating" $B$ via grouplikes (See Lemma \ref{B tensor G}).

%We first classify $\Ext$-quivers of the comodule category $\Mm^B$ in Theorem \ref{ext quiver}. It turns out that the $\Ext$-quivers is either of types  (1) a single point; (2) an oriented cycle $\widetilde A_n$; (3) double infinite quiver $_\infty A_\infty$; which corresponds to $B$ being a coserial Hopf algebra, or (4) a certain quiver which satisfies for each vertex there are exactly two arrows coming in and going out.
%Since the coserial Hopf algebras has been classified in \cite{IJ}, we will focus on the non-coserial case only.    
%Next, we determine the algebra structure of a non-coserial link indecomposable discrete corepresentation type Hopf algebra $B$ in Theorem \ref{B is discrete}.
% Up to isomorphism, such a Hopf algebra is isomorphic to one of the algebra in the list of Theorem \ref{classify B}.

{%\setlength{\parindent}{5pt}
\section{Preliminaries}\label{Sec:preliminaries}

We recall a few basic facts about Hopf algebras, representations and comodule categories. We refer the reader to the textbooks \cite{DNR,R,M} for further details.

\subsection{Pointed coalgebras and coradical filtrations}
A quiver $Q=(Q_0,Q_1)$ is a directed graph with $Q_0$ the set of vertices and $Q_1$ the set of arrows.
 For any arrow $\alpha\in Q_1$, denote by $s(\alpha)$ for the source of the arrow and $t(\alpha)$ the target of the arrow. The path coalgebra $\KK Q$ of a quiver $Q$ over a field $\KK$ is defined as follows. As a vector space, it   has a basis consisting of all the paths in $Q$. The comultiplication $\Delta$ is defined as

$$\Delta(p)=\sum\limits_{p=(q|r)}q\otimes r$$
where $(q|r)$ is the concatenation of paths $q$ followed by $r$ and the sum is done over all the possible ways of splitting the path $p$ as a concatenation of two subpaths. The co-unit of the path coalgebra $\KK Q$ is $\varepsilon(p)=\delta_{{\rm length}(p),0}$. In this paper, the path algebra of the quiver $Q$ will be denoted by $\KK[Q]$ to differentiate it from the path coalgebra. We note that the path algebra and path coalgebra structures are usually not compatible in any bialgebra way, but rather dual to each other (see also \cite{DIN,DIN2} for relations between the path algebra and path coalgebra, and further details on these). In particular, as we always consider a Hopf algebra $H$ as a sub-coalgebra of a certain path coalgebra, the path $(q|r)$ is different from the multiplication $r\cdot q$ or $q\cdot r$ in the Hopf algebra $H$.  In general, we denote by $(\alpha_1 |\alpha_2|\cdots|\alpha_n)$ the path $\cdot\stackrel{\alpha_1}\rightarrow \cdot\stackrel{\alpha_2}\rightarrow \cdots \stackrel{\alpha_n}\rightarrow \cdot$.

%Throughout, we will thus choose and fix an embedding of $H$ into $\KK Q$.

%We recall a few facts about quiver representations and the category left $\KK[Q]$-modules. Let $Q$ be finite quiver. Denote by $R_Q$ the arrow ideal of $Q$. An admissible ideal $I$ of the path algebra $\KK[Q]$ is a two sided ideal such that $R^m_Q\subseteq I\subseteq R^2_Q$ for some integer $m$. A bounded quiver algebra is $\KK[Q]/I$ for some finite quiver $Q$ and some admissible ideal $I$. It is well-known that the category $\KK[Q]/I\Modd$ of left $\KK[Q]/I$-modules (respectively  $\KK[Q]/I\modd$ of finitely generated left $\KK[Q]/I$-modules) is equivalent to the category $\Rep(Q,I)$ of representation of the quiver $Q$ with ideal $I$ (respectively  $\rep(Q,I)$ of finite dimensional representations of the $Q$ with ideal $I$).
%\begin{remark}
%If $Q$ is a finite acyclic quiver, then there are inclusions of categories $$\Rep(Q,R^2_Q)\subseteq\Rep(Q,I)\subseteq \Rep(Q)$$
% $$\rep(Q,R^2_Q)\subseteq\rep(Q,I)\subseteq \rep(Q).$$
%\end{remark}
%
%
%
%
Denote by ${}^{C}\Mm$ (resp. $\Mm^C$) the category of left (resp. right) comodules over the coalgebra $C$. We also recall a few facts about the category ${}^{\KK Q}\Mm$.    By results of \cite{CKQ}, (see also \cite{DIN}), the category ${}^{\KK Q}\Mm$ of left $\KK Q$-comodules is equivalent to the category of locally nilpotent right $\KK[Q]$-modules (equivalently, locally nilpotent representations of the opposite quiver of $Q$). We recall that a (right) module $M$ is locally nilpotent if given any $x\in M$, there exists a monomial ideal $I$ of $\KK[Q]$ of finite codimension such that $Ix=0$. (See also \cite{si1}). In particular, if $Q$ is a finite acyclic quiver, then ${}^{\KK Q}\Mm$ is equivalent to the category of right $\KK[Q]$-modules (equivalently $\Rep(Q^{op}))$.

\begin{definition}\label{Ext-quiver}
 Let $C$ be a pointed $\KK$-coalgebra with a group of grouplikes $G=G(C)$. %The grouplikes correspond precisely to the 1-dimensional simple $C$-comodules.
 Denote by $P(g,h)$  the space of $g$-$h$ skew primitive elements: $P(g,h)=\{x\in C | \Delta(x)=g\otimes x + x\otimes h\}$. The {\bf $\Ext$-quiver} $Q$ of the comodule category $\Mm^C$ of $C$ is  defined as follows: if $g,h\in G(C)$ are distinct grouplikes in $C$, then they represent distinct vertices in $Q$; the number $n$ of arrows from $g$ to $h$ is equal to $\dim(P(g,h))-1$. 
 %\margin{Define the Ext quiver for coalgebras instead of Hopf algebras.}
\end{definition}

By results of \cite{CM,DNR} which are dual to classical results of Gabriel for finite dimensional algebras, any pointed coalgebra embeds in a quiver coalgebra, where the quiver can be chosen to be the Ext-quiver of the comodule category $\Mm^H$ of $H$:

\begin{proposition}
 If $C$ is a finite dimensional coalgebra over an algebraically closed field $\KK$ and  $Q$ is the $\Ext$-quiver of $C$, then $C$ is isomorphic to a subcoalgebra of the path coalgebra $\KK Q$. At the same time, $C^*$ is a finite dimensional algebra, which is isomorphic to $\KK[Q^{op}]/I$ for some admissible ideal $I$.
 \end{proposition}

For a finite dimensional pointed coalgebra $C$,  its coradical filtration is a series of sub-coalgebras
$$
0\subseteq C_0\subseteq C_1\subseteq \cdots\subseteq C_n=C
$$ defined inductively by $C_0=\KK G(C)$ consisting of linear combinations of grouplikes and  $x\in C_{i+1}$ if $\Delta(x)=C_i\otimes C+C\otimes C_0$. In particular, $C_1$ is a sub-coalgebra of $C$ which consists all the grouplikes and skew-primitives. The length of the coradical filtration is called the Loewy length of $C$. The coradical filtration of $C$ is a dual of the radical filtration of $C^*$.

\subsection{Corepresentation discrete type coalgebras}
The following theorem is usually referred as the $2^{nd}$ Brauer-Thrall conjecture, which has been proved to hold true.

\begin{theorem}
A finite dimensional algebra over an algebraically closed field $\KK$ is either representation-finite type or there exists an infinite sequence of numbers $d_i\in\mathbb N$ such that for each $i$, there exists an infinite number of non-isomorphic indecomposable modules with $\KK$ dimensional $d_i$.
\end{theorem}

We omit the dual statement of Brauer-Trall conjecture for finite-dimensional coalgebras.
We will heavily use the following lemma when proving statements about finite/discrete type. Note that a quiver is called {\bf bipartite} if each vertex is either a sink or source.
\begin{lemma} \label{non discrete type}
Let $H$ be a Hopf algebra and $Q$ the $\Ext$-quiver of $\Mm^H$. Suppose there is a finite subquiver $Q'\subseteq Q$ satisfying the sub-coalgebra $\KK Q'_1$ (in the coradical filtration) is infinite corepresentation type, then $H$ is not of  corepresentation discrete  type. In particular, if $Q'$ is a bipartite non-Dynkin type quiver, then $H$ is not discrete corepresentation type.
\end{lemma}

\begin{proof}
Since $Q$ is the Ext-quiver of $\Mm^H$,  $\KK {Q'}_1\subseteq \KK Q_1\subseteq H$ are extensions of sub-coalgebras. So there is an inclusion $\Mm^{\KK Q'_1}\to \Mm^H$ of comodule categories. But by assumption $\KK Q'_1$ is a finite dimensional coalgebra of infinite corepresentation type, so there is an infinite family of  $\KK Q'_1$ comodules with a dimension vector $\underline d$.  Therefore $\Mm^H$ also contains an infinite family of comodules with a dimension vector $\underline d$, hence not of discrete corepresentation type.

For the ``in particular'' case: notice that for a bipartite quiver $Q'$, $\KK Q'_1=\KK Q'$.
\end{proof}

As an immediate corollary,  the $\Ext$-quiver $Q$ of a discrete corepresentation type Hopf algebra is {\bf Schurian}, that is no more than one arrow between any two vertices, because otherwise there will be a Kronecker 2-quiver as a subquiver of $Q$.

In order to introduce a formal way to relate the category of finite dimensional representations of a quiver, (or a coalgebra of $\KK Q$) to abelian subcategories which are equivalent to categories of modules over some finite dimensional algebra, we will need to use the notion of coefficient coalgebra of an $C$-comodule, where $C$ is any coalgebra. Given a left $C$-comodule $M$, with comultiplication map $\rho:M\rightarrow C\otimes M$, its coefficient coalgebra $cf(M)$ is defined as the smallest subcoalgebra of $C$ such that $\rho(M)\subseteq cf(C)\otimes M$; see \cite[Chapter 2]{DNR}. We note that this is nothing else but the dual analogue of the annihilator; indeed, if $C$ is a finite dimensional coalgebra, the category of $C$-comodules is equivalent to that of right modules over the dual algebra $C^*$. If $M$ is a left $C$-comodule, then $cf(M)=\{x\in M | f(x)=0, \,\forall \,f\in {\rm ann}_{C^*}(M)\}$. In general, over any coalgebra, if $M$ is a left $C$-comodule, it is automatically a right $C^*$-module, and the annihilator of $M$ is ${\rm ann}_{C^*}M=\{f\in C^* | f(c)=0,\,\forall\, x\in cf(M)\}$. Given any dimension vector $\underline{d}$, denote by $cf(\underline{d})$ the smallest sub-coalgebra of $C$ such that all the $C$-comodules of dimension vector $\underline{d}$ have their coefficient coalgebra contained in $cf(\underline{d})$ \cite{IP}.

%Let $C$ be a locally finite coalgebra, that is, its $\Ext$-quiver is locally finite in the sense that $Ext(S,T)$ is finite dimensional for any two simple left comodules $S,T$. In this case,

\begin{proposition}\label{discrete type dual}
If $C$ is any coalgebra, then $C$ is of discrete corepresentation type if and only if every finite dimensional subcoalgebra $D$ of $C$ is of finite corepresentation type.
\end{proposition}
\begin{proof}
%First, since any coalgebra $C$ is Morita equivalent to a pointed coalgebra, we may reduce to the case when $C$ is pointed. Then, observe that the $\Ext$-quiver of $Q$ is Schurian in both cases.  Next, since $C$ (and its $\Ext$-quiver $Q$) is locally finite (the arrow set between any two vertices is finite), we note that given any dimension vector $\underline{d}$, the coefficient coalgebra $cf(\underline{d})$ is finite dimensional \cite[Lemma 2.6]{IP}.  Therefore, the finiteness of the set of isomorphism classes of dimension $\underline{d}$ corepresentations for all $\underline{d}$ reduces equivalently to the the same property for every finite dimensional subcoalgebra $D$ of $C$ and hence to that every such $D$ is of finite representation type.
Assume every finite dimensional subcoalgebra $D$ of $C$ is finite corepresentation type. We note that given any dimension vector $\underline{d}$, the coefficient coalgebra $cf(\underline{d})$ is finite dimensional \cite[Lemma 2.6]{IP}, hence finite representation type. Therefore, the set of isomorphism classes of dimension $\underline{d}$ corepresentations is finite. That is, $C$ has discrete corepresentation type.
Conversely, assume there is a finite dimensional subcoalgebra $D\subseteq C$  which is infinite corepresentation type. Then Brall-Thrall conjecture asserts that there exists infinite number of non-isomorphic indecomposable $D$-comodules with dimension vector $\underline{d}$. Therefore, $C$ is not discrete corepresentation type. 
\end{proof}

\subsection{Decompositions of pointed Hopf algebras}
In what follows, all Hopf algebras considered will be pointed.  Let $Q$ be the $\Ext$-quiver of the comodule category $\Mm^H$. In general, $Q$ is not necessary a connected quiver. We are going to introduce the notion of link-indecomposable components following \cite{Mon}:

\begin{definition}\label{LIC}
Let $C$ be a pointed coalgebra. For two simple coalgebras $S$ and $T\subseteq C$, the wedge $S\wedge T= \Delta^{-1}(S\otimes C+C\otimes T)$. Define a quiver $\Gamma_C$ as the following:
\begin{enumerate}
\item the vertices are simple subcoalgebras of $C$,
\item there exists and arrow $S\to T$ if and only if $S\wedge T\neq S+T$.
\end{enumerate}
A {\bf link-indecomposable-component (L.I.C.)} $D$ is a subcoalgebra $D\subseteq C$ which is maximal with respect to $\Gamma_D$ being connected (as an undirected graph).
\end{definition}

\begin{remark}
It is known that for a pointed coalgebra $C$, $S$ is a simple subcoalgebras if and only if $S=\KK g$ for some $g\in G(C)$ and $S\wedge T\neq S+T$ for simple subcoalgebra $S=\KK g$ and $T=\KK h$ if and only if there is a non-trivial skew-primitive in $P(g,h)$. Therefore the quiver $\Gamma_C$ can be obtained by identifying multiple arrows of the $\Ext$-quiver of the comodule category $\Mm^C$ and a L.I.C. $D$ is a subcoalgebra which is maximal with respect to the $\Ext$-quiver of the comodule category $\Mm^D$ being connected.
\end{remark}

We have the following decomposition theorem of pointed Hopf algebras.
\begin{theorem} \label{B tensor G}\cite[Theorem 3.2]{Mon}
 Let $H$ be a pointed Hopf algebra and $B$ be the L.I.C. containing $1$. Then,
 \begin{enumerate}
 \item $B$ is a Hopf algebra.
 \item $G(B)$ is a normal subgroup of $G(H)$;
 \item $G(H)$ acts on $B$ by conjugation.
 \item  $H\cong B\#_\sigma \KK(G(H)/G(B))$ is a crossed product of $B$ and the group algebra $\KK(G(H)/G(B))$, where $\sigma:G(H)/G(B)\times G(H)/G(B)\to G(B)$ is a $2$-cocycle.
 \end{enumerate}
\end{theorem}
 According to the theorem, $H\cong B\otimes_\KK \KK(G(H)/G(B))$ as a coalgebra.

\section{Separable extensions and representation finiteness}
In this section, we will discuss the main techniques of determining the (co)representation finiteness of a (co)algebra. It turns out that techniques using separable extensions of algebras serve as a fundamental theorem. So we start from recalling some results about algebra extensions from \cite{IS}.

\subsection{Separable extensions}
A covariant functor $F:\Cc\to \Dd$ induces a natural transformation between bifunctors $\bar{F}:\Cc(-,-)\to \Dd(F-,F-)$ sending $\alpha:X\to Y$ to $F(\alpha)$. If $\bar{F}$ admits a natural section, then $F$ is said to be separable. Let $\KK$ be a field and $A$, $B$ are finite dimension algebras over $\KK$. An algebra extension $A\subseteq B$ is called {\bf separable} if the restriction $ B\modd\to A\modd$ is a separable functor. For a separable extension $A\subseteq B$, $A$ is called a separable subalgebra of $B$. The following is a useful criteria of separable subalgebras.

\begin{lemma}\cite{CZ, K, IS}\label{central idempotent}
 A finite dimensional algebra $B$ is a separable extension of $A$ if and only if the multiplication map $u: B\otimes _A B\to B $ is a split epimorphism of $B$-bimodules, if and only if there is an element $e\in B\otimes_A B$ such that  the multiplication map $u:B\otimes_A B\to B$ satisfies $u(e)=1_A$ and $eb=be$ for all $b\in B$.
\end{lemma}

The importance of separable extension is that it allows us to transfer representation finiteness from the subalgebra.

\begin{lemma}\cite[Corollary 1.8]{IS}\label{sep fin type}
If $A\subseteq B$ is a separable extension of finite dimensional algebras, then every indecomposable $B$-module is a direct summand of a module induced from an indecomposable $A$-module. In particular, if $A$ is  finite representation type, then so is $B$.
\end{lemma}

We can apply the concept of separable extensions in the context of coalgebras using so-called covering maps which are introduced below.

%%%%%%%%%%%%%%%%%%%%%%%%%%%%%%%%%%%

\subsection{Covering maps} Let $C\subseteq \KK Q$ be a sub-coalgebra. We call an element $d\in C$ is a {\bf diamond}, if $d$ is a linear combination of paths with a common source, denoted by $s(d)$ and a common sink, denoted by $t(d)$.

\begin{proposition} \cite{JMR}
Let $C\subseteq \KK Q$ be a sub-coalgebra. Then
there exists a $\KK$-linear basis of $C$ such that each element in the basis is a diamond.
\end{proposition}

We call such a basis in the above proposition a {\bf diamond basis}.

\begin{definition}\label{covering def}
Let $C\subseteq \KK Q$ and $D\subseteq \KK Q'$ be finite dimensional pointed coalgebras with diamond basis $B$ and $B'$ respectively. Suppose $Q_0\cup Q_1\subseteq B$ and $Q'_0\cup Q'_1\subseteq B'$.  A coalgebra homomorphism $\pi: C\to D$ is called {\bf a covering map} if
\begin{enumerate}
\item $\pi(B)=B'$.
\item Let $b_1, b_2\in B$ satisfying $\pi(b_1)=\pi(b_2)$. Then $s(b_1)=s(b_2)$ or $t(b_1)=t(b_2)$ implies $b_1=b_2$.
\end{enumerate}

\end{definition}

\begin{remark}
(1) A covering map $\pi$ is surjective satisfing $\pi(Q_0)=Q'_0$, $\pi(Q_1)=Q'_1$.\\ % dual basis is called b*
(2) Condition (2) is equivalent to say that if $b_1$ and $b_2$ are distinct diamonds in $B$ with the same sources or targets, then $\pi(b_1)\neq\pi(b_2)$. Sometimes it is convenient to use this equivalent condition to check if a coalgebra homomorphism is a covering map. (See Example \ref{exm:covering})
\end{remark}

 Covering maps allows us to transfer representation finiteness between two coalgebras.
\begin{lemma} \label{covering fin type}
If $\pi:C\to D$ is a covering map, then $\pi^*: D^*\to C^*$ is a separable extension of algebras. Hence if  $D$ is corepresentation finite, then so is $C$.
\end{lemma}
\begin{proof}
We prove the lemma using Lemma \ref{central idempotent}.
   Suppose $e^*$'s are primitive idempotents of $D^*$. Then $\sum e^*\otimes_{D^*} e^*$ is an element in $C^*\otimes_{D^*}C^*$ such that under the multiplication $u(\sum e^*\otimes_{D^*} e^*)=1$. It suffice to show that for any diamond $x$ in the diamond basis of $C$, $x^*$ commutes with $\sum e^*\otimes_{D^*} e^*$.

Since $\pi$ is a covering map, $\pi^*(\pi(x)^*)=\sum\limits_{\pi(x_i)=\pi(x)}x_i^*$, where the diamonds $x_i$ have mutually different sources and targets. Notice that $\sum\limits_{\pi(x_i)=\pi(x)}x_i^*\in D^*$ since $\pi^*$ is a monomorphism.

% $\pi^*(\tilde e^*)=\sum\limits_{\pi(e_i)=\tilde e}e_i^*$,  $\pi^*(\tilde x^*)=\sum\limits_{\pi(x_i)=\tilde x}x_i^*$.

 Hence $x^*\sum e^*\otimes_{D^*} e^*=\sum x^*e^*\otimes_{D^*} e^*= t(x^*)x^*\otimes_{D^*} s(x^*)= t(x^*)\pi^*(\pi(x)^*) \otimes_{D^*} s(x^*)= t(x^*) \otimes_{D^*} \pi^*(\pi(x)^*)s(x^*)= t(x^*) \otimes_{D^*} x^*=\sum e^*\otimes_{D^*} e^*x^*$.

  Therefore, by Lemma  \ref{central idempotent}, $\pi^*: D^*\to C^*$ is a separable extension. The second statement then follows immediately from Lemma \ref{sep fin type}.
\end{proof}

\begin{example}\label{exm:covering}
Let $C$ be the path coalgebra $\KK \tilde Q$ of the following quiver $\tilde Q$ and $D$ a subcoalgebra of the path coalgebra $\KK Q$ of the following quiver $Q$, which is linearly generated by paths $Q_0\cup Q_1\cup \{(\alpha|\beta), (\beta|\gamma)\}$

\begin{center}
\begin{tikzpicture}[->]
\node(1) at (0,0) {$1$};
\node(2) at (0,1) {$2$};
\node(4) at (1,1) {$4$};
\node(3) at (1,0) {$3$};
\node (c) at (-1,0.5) {$\tilde Q$:};
\draw(1)--node[left]{$\tilde\beta$}(2);
\draw(2)--node[above]{$\tilde\gamma$}(4);
\draw(3)--node[right]{$\tilde\delta$}(4);
\draw(1)--node[below]{$\tilde\alpha$}(3);

 \node(5) at (3,0) {$1$};
\node(6) at (3,1) {$2$};

\node (d) at (2.2,0.5) {$Q$:};
\draw(5)--node[left]{$\beta$}(6);
 \draw [->] (5.east)arc(150:-150:0.15)node[right]{\ \ $\alpha$};
\draw [->] (6.east)arc(150:-150:0.15)node[right]{\ \ $\gamma$};
\end{tikzpicture}
\end{center}

It is obvious that $C$ has a diamond basis consisting of all the paths in $\tilde Q$ and $D$ has a diamond basis $Q_0\cup Q_1\cup \{(\alpha|\beta), (\beta|\gamma)\}$. Consider the coalgebra homomorphism $\pi: C\to D$ defined on basis by
$$
\pi(\tilde e_1)=\pi(\tilde e_3)=e_1, \pi(\tilde e_2)=\pi(\tilde e_4)=e_2;
$$
$$
\pi(\tilde \alpha)=\alpha, \pi(\tilde \beta)=\pi(\tilde \delta)=\beta, \pi(\tilde \gamma)=\gamma;
$$
$$
\pi((\tilde\alpha|\tilde\delta))=(\alpha|\beta); \pi((\tilde\beta|\tilde\gamma))=(\beta|\gamma).
$$
It is easy to see that $\pi$ is a covering map. In fact, the only pairs of diamonds in the diamond basis of $C$ which have the same sources are $\tilde\alpha$ and $\tilde\beta$; $(\tilde\alpha|\tilde\delta)$ and $(\tilde\beta|\tilde\gamma)$. But neither of them are mapped to the same diamond in $D$ via the map $\pi$. Similarly, one can check for pairs of diamonds with the same targets.

Since $C$ is infinite representation type, by Lemma \ref{covering fin type}, so is $D$.
\end{example}

We also want to introduce a special class of covering maps which are constructed from quotient quivers. Let $Q=(Q_0,Q_1)$ and $Q'=(Q'_0,Q'_1)$ be quivers. A {\bf quiver homomorphism} $q: Q\to Q'$ is a map $q(Q_0)\subseteq Q'_0$ and $q(Q_1)\subseteq Q'_1$ such that for each $\alpha\in Q_1$, $s(q(\alpha))=q(s(\alpha))$ and $t(q(\alpha))=q(t(\alpha))$.

A quotient quiver $\bar Q=(\bar Q_0, \bar Q_1)$ of $Q=(Q_0,Q_1)$ is a quiver, whose vertices are blocks of partitions of $Q_0$ and the number of arrows from $\bar u$ to $\bar v$ in $\bar Q$ equals the total number of arrows from $u$ to $v$ for all $u\in \bar u$ and $v\in \bar v$. A quotient quiver $\bar Q$ can be obtained by gluing vertices in the same block of $Q$ and preserve all the arrows at the same time. It is easy to see that there is a canonical homomorphism $q: Q\to\bar Q$, of quotient quivers, sending vertices $v\mapsto \bar v$ and arrows $(u\stackrel{\alpha}\to v) \mapsto (\bar u\stackrel{\bar\alpha}\to \bar v)$. It is easy to see that $q$ sends any path $(\alpha_1|\cdots|\alpha_n)$ in $Q$ to the path $(\bar\alpha_1|\cdots|\bar\alpha_n)$ in $\bar Q$. Hence  $q$ can be extended to a linear map $q:\KK Q\to \KK \bar Q$.

For any subcoalgebra $C$ of $\KK Q$, define $q(C)$ the smallest subcoalgebra of $\KK \bar Q$ containing $q(c)$ for all $c\in C$. Notice that $q(Q_0)=\bar Q_0$ and there is a one-to-one correspondence between arrows $\alpha$ in $Q$ and arrows $\bar \alpha$ in $\bar Q$. For any non-trivial path $\bar p$ in $\bar Q$, there is a unique path $p$ in $Q$, such that $q(p)=\bar p$. Hence $q:C\to q(C)=\{q(c)|c\in C\}$ is actually a coalgebra homomorphism.

%Hence, given any subcoalgebra $C$ of $\KK Q$, define $q(C)=\{ \sum c_iq(p_i) | \sum c_ip_i\in C \}$ as a vector space in $\KK\bar Q$. Define $\Delta_{q(C)}(g)=g\otimes g$ for $g\in q(Q_0)$ and $\Delta_{q(C)}(q(p))=q\otimes q\Delta_C(p)$  for paths $p$ with length at least $1$.

\begin{lemma}\label{quotient quiver}
Let $q:Q\to \bar Q$ be a homomorphism of quotient quivers. For any subcoalgebra $C\subseteq \KK Q$, (1) For any diamond basis $B$ of $C$, $\{q(b)|b\in B\}$ is a diamond basis of $q(C)$ (2) $q:C\to q(C)$ is a covering map.
\end{lemma}
\begin{proof}
(1) is clear by definition.  For (2),  Let $b_1=\sum c_ip_i$, $b_2=c'_ip'_i\in B$ be diamonds satisfying  $ {q}(b_1)= {q}(b_2)$, where $p_i$, $p'_i$ are paths. First assume all $p_i$'s are paths of length at least $1$. Since $q$ preserves the length of paths, it follows that $p'_i$'s are also paths of lengths at least $1$. Then it forces $b_1=b_2$, because $q$ is a bijection between the set of non-trivial paths on $Q$ and $\bar Q$. Now  assume $p_1=e$ is a vertex and $p_i$ are paths of length at least $1$ for $i>1$. Since $b_1$ is a diamond, $s(b_1)=e=t(b_1)$. Because $ {q}(b_1)= {q}(b_2)$, $b_2$ contains a term $e'$ such that  $q(e)=q(e')$ and hence $s(b_2)=e'=t(b_2)$. So we can write $b_2=c'_1e'+\sum_{i>1}c'_ip'_i$. If in additional $s(b_1)=s(b_2)$ or $t(b_1)=t(b_2)$, then $e=e'$ and $b_1=b_2$. Therefore ${q}$ is a covering map.
\end{proof}

\begin{corollary}\label{cor quotient quiver}
Let $C$ be a pointed coalgebra with $\Ext$-quiver $Q$ of $\Mm^C$. Suppose there is a homomorphism of quotient quivers $q:Q\to\bar Q$.   If $q(C)$ has finite corepresentation type, then so does $C$.
% In particular, if $Q$ is a bipartite non-Dynkin type quiver, then $D$ is infinite representation type.
  \end{corollary}
%\begin{proof}
%(1) Since the covering map $q$ preserves vertices and arrows, $\bar Q_0=q(Q_0)=G(D)$. And for any $\bar g,\bar h\in G(D)$, $\dim P(\bar g,\bar h)=\sum\limits_{p(g)=\bar g, p(h)=\bar h}\dim P(g,h)$.

%(2) It follows immediately from  Lemma \ref{covering fin type} and Lemma \ref{quotient quiver}.

%In particular, when $Q$ is a bipartite non-Dynkin type quiver, $\KK Q=\KK Q_1\subseteq C$. So $C=\KK Q$ is a infinite representation type coalgebra. By (2), $D=q(C)$ is also a  infinite representation type coalgebra.
%\end{proof}

\begin{lemma}\label{lem:quotient quiver}
Let $C$ be a pointed coalgebra. Suppose the $\Ext$-quiver of $\Mm^C$ contains a subquiver $Q'$ such that there is a homomorphism of quotient quivers $q:Q\to Q'$ for some finite bipartite non-Dynkin type quiver $Q$, then $C$ is not discrete corepresentation type.
  \end{lemma}
\begin{proof}
The quiver homomorphism induces a covering map: $q:\KK Q=\KK Q_1\to q(\KK Q_1)=\KK Q'_1$. By Corollary \ref{cor quotient quiver}, $\KK Q'_1$ is infinite corepresentation type. At the same time, $\KK Q'_1\subseteq C_1\subseteq C$ is a finite dimensional  subcoalgebra of $C$. Therefore $C$ is not discrete corepresentation type.
\end{proof}

\begin{example}\label{no 3 arrows}
A coalgebra $C$ with the following $\Ext$-quiver of $\Mm^C$ has infinite corepresentation type.
\begin{center}
\begin{tikzpicture}[->]
\node(1) at (0,0) {$1$};
\node(2) at (0,1) {$2$};
\node(4) at (1,1) {$4$};
\node(3) at (1,0) {$3$};
\node (c) at (-1,0.5) {$Q'$:};
\draw(1)--(2);
\draw(2)--(4);
\draw(3)--(4);
\draw(1)--(3);

\draw [->] (2.west)arc(330:30:0.15);
\draw [->] (3.east)arc(150:-150:0.15);
\end{tikzpicture}
\end{center}

Notice that $Q'$ is a quotient quiver of  the following bipartite quiver (by gluing vertices $2$ with $2'$ and $3$ with $3'$).
\begin{center}
\begin{tikzpicture}[->]
\node(1) at (-.5, 0.8) {$1$};
\node(2) at (0.5,0.8) {$2$};
\node(3) at (1,0) {$2'$};
\node(4) at (0.5,-0.8) {$4$};
\node(5) at (-0.5,-0.8) {$3'$};
\node(6) at (-1,0) {$3$};
 \node at (-1.5,0) {$Q$:};

 \draw[->] (1)--(2);
  \draw[->] (3)--(2);
  \draw[->] (3)--(4);
   \draw[->] (5)--(4);
    \draw[->] (5)--(6);
    \draw[->] (1)--(6);
\end{tikzpicture}
\end{center}

Hence $C$ is infinite corepresentation type by Lemma \ref{lem:quotient quiver}.
\end{example}

\subsection{Localization}

Let $A$ be a finite dimensional algebra and $e=e^2\in A$ be an idempotent. The localization is a  functor $L: A\modd\to eAe\modd$ sending $M$ to $Me=M\otimes_A Ae$. It has a quasi-inverse functor $R=\Hom_{eAe}(Ae,-)$. Hence we have the following well-known fact:

\begin{proposition}\label{localization}
If $A$ is representation finite, then so is the localization $eAe$.
\end{proposition}

\begin{example}\label{localization exm}
Let $Q$ be the following quiver and $I$ denote the $\lambda$-commutative relations of two squares $yx+\lambda vu=0$ for some $\lambda\in\KK^\times$. Then $\KK[Q]/I$ is infinite representation type.
$$\begin{tikzpicture}[->]
\node(0) at (-1,3) {$a$};
\node(1) at (-1,2) {$\cdot$};
\node(2) at (0,2) {$\cdot$};

\node(4) at (0,3) {$\cdot$};
\node(5) at (0,1) {$\cdot$};
\node(6) at (1,1) {$\cdot$};
\node(66) at (1,0) {$\cdot$};

\node(7) at (1,0) {$\cdot$};
\node(8) at (2,0) {$\cdot$};
\node(9) at (1,-1) {$\cdot$};
\node(10) at (2,-1) {$b$};

\draw(4)--node[above]{$x$}(0); \draw(0)--node[left]{$y$}(1); \draw(8)--node[right]{$u$}(10); \draw(10)--node[below]{$v$}(9);
\draw(2)--node[below]{$v$}(1); \draw(4)--node[right]{$u$}(2);\draw(2)--(5);\draw(6)--(5);
\draw(6)--(7); \draw(8)--node[above]{$x$}(7); \draw(7)--node[left]{$y$}(9);
\end{tikzpicture}$$

In fact, consider the localization $eAe$, where $e=1-e_a-e_b$. The localization $eAe$ is the path algebra of a $\tilde{D}_7$ type quiver, hence infinite representation type. So $A$ must also be infinite representation type.

\end{example}

\section{The $\Ext$-quiver of pointed Hopf algebras of discrete corepresentation type }

%This is an outline of our paper. It is included here only to give us a roadmap, but will likely not be in the paper as we start transforming it into more polished work. Note that we will be using the above notation, meaning we fix an embedding of $H$ into $\KK Q$.

Throughout this section, let $H$ be a pointed Hopf algebras of discrete corepresentation type. We are going to determine the $\Ext$-quiver of the comodule category $\Mm^H$.

 \subsection{}
We start from studying the number of arrows coming in (going out of) each vertex of the $\Ext$-quiver $Q$.
Let $v$ be a vertex in a quiver $Q$. Define the star quiver $Star(v)$ as the subquiver of $Q$ whose vertices are the vertex $v$ and all the vertices adjacent to $v$; whose arrows are all the arrows start or terminate in $v$.

\begin{lemma}
If $H$ is a pointed Hopf algebra and $Q$ is the $\Ext$-quiver of $\Mm^H$. Then for group-like elements $g$ and $h$, there is an isomorphism of quivers $Star(g)\cong Star(h)$.
\end{lemma}
\begin{proof}
It suffice to show that $Star(1)\cong Star(g)$ for all $g$. \\
Denote by $A(g,h)$ the set of arrows from $g$ to $h$ in $Q$. Denote by $Out_g=\{h\neq g : |A(g,h)|\neq 0\}$ and $In_g=\{h\neq g : |A(h,g)|\neq 0\}$.  Since for any group-like elements $g,h\in H$, $c:x\mapsto gx$ is a map from $P(1,h)$ to $P(g,gh)$ with an inverse $d: x\mapsto g^{-1}x$. The sets $P(1,h)$ and $P(g,gh)$ have the same cardinality. Hence there is a bijection $\varphi^{out}_h: A(1,h)\to A(g,gh)$. In particular, $h\in Out_1$ if and only if $gh\in Out_g$. Similarly, there are bijections $\varphi^{in}_h: A(h,1)\to A(gh,g)$ and $\varphi^{loop}: A(1,1)\to A(g,g)$.\\
 It is easy to see that there is a quiver isomorphism $\varphi: Star(1)\to Star(g)$ which sends each vertex $h$ to $gh$ and sends an arrow $\alpha: s\to t$ by
$$
\varphi(\alpha)=
\begin{cases}
\varphi^{out}_t(\alpha) & 1\neq t\in Out_1\\
\varphi^{in}_s(\alpha) & 1\neq s\in In_1\\
\varphi^{loop}(\alpha) & s=t=1.\\
\end{cases}
$$
\end{proof}

The following Proposition is an immediate corollary:

\begin{proposition}\label{homogeneous}
If $H$ is a pointed Hopf algebra, then the $\Ext$-quiver $Q$ of $\Mm^H$ is homogeneous in the following sense: there exist $n,p$ such that at every vertex, exactly $n$ arrows go out (to a different vertex), exactly $n$ come in, and exactly $p$ loops exist. Here, $n,p$ could be any cardinalities, although in our case they will be small numbers.
\end{proposition}

Furthermore, we have the following:

\begin{lemma}
At each vertex $v$, the number of arrows going in $v$ equals the number of arrows coming out of $v$.
\end{lemma} %cardinal

\begin{proof}
Notice that a loop always contributes to both an arrow going in and an arrow coming out of a vertex. So together with Proposition \ref{homogeneous},  we just need to prove the number going in $1$ from a distinct vertex equals the number of arrows coming out of $1$ to a distinct vertex.

In fact, if $x\in P(1,g)$ for some group like element $g\neq 1$, then $S(x)\in P(g^{-1},1)$, where $S$ is the antipode. Since $S$ is invertible, $\dim P(1,g)=\dim P(g^{-1},1)$. The statement follows.
\end{proof}

\subsection{}
In addition, since $H$ is discrete corepresentation type, the $\Ext$-quiver $Q$ is Schurian. Furthermore, we are going to show that more restrictions need to be put on the number of arrows going out/coming in at each vertex in $Q$.

\begin{lemma}\label{ab=ba}
Let $H$ be a discrete corepresentation type Hopf algebra. If there are group-like elements $a\neq b$, such that $x\in P(1,a)$ and $y\in P(1,b)$ are non-trivial skew primitives, then $ab=ba$.
\end{lemma}
\begin{proof}
If $a=1$ or $b=1$, the statement holds obviously. Now assume $a\neq 1$, $b\neq 1$ and $ab\neq ba$, then there are arrows:
$$
\begin{tikzpicture}[->]
\node (a) at (0,2) {$a$};
\node (b) at (2,2) {$b$};
\node (ab) at (0,0) {$ab$};
\node(ba) at (2,0) {$ba$};

\draw (a)--node[left]{$ay$}(ab);
\draw(a)--node[left]{$ya$}(ba);
\draw(b)--node[right]{$xb$}(ab);
\draw(b)--node[right]{$xa$}(ba);

\end{tikzpicture}$$
Since four vertices $a$, $b$, $ab$ and $ba$ are mutually distinct, this $\tilde A_3$ type quiver is a sub-quiver of the $\Ext$-quiver of $\Mm^H$, which is a contradiction due to Lemma \ref{non discrete type}.
\end{proof}

%This forms a subcoalgebra which is a subcoalgebra of the path coalgebra of \tilde A3

\begin{proposition}\label{two arrows}
Let $H$ be a Hopf algebra of discrete representation type and $Q$ be the $\Ext$-quiver of $\Mm^H$. Then for each vertex of $Q$, there are at most $2$ arrows going out (and respectively at most 2 arrows coming in).
\end{proposition}
\begin{proof}
According to Proposition \ref{homogeneous}, it suffice to prove for the vertex $1$.
If we have  3 outgoing arrows from $1$ say to $a,b,c$, then consider the following diagram

\begin{center}
\begin{tikzpicture}[->]
\node(1) at (-.5, 0.8) {$a$};
\node(2) at (0.5,0.8) {$ab$};
\node(3) at (1,0) {$b$};
\node(4) at (0.5,-0.8) {$bc$};
\node(5) at (-0.5,-0.8) {$c$};
\node(6) at (-1,0) {$ca$};

 \draw[->] (1)--node[above]{\tiny$ay$}(2);
  \draw[->] (3)--node[right]{\tiny$xb$}(2);
  \draw[->] (3)--node[right]{\tiny$bz$}(4);
   \draw[->] (5)--node[below]{\tiny$yc$}(4);
    \draw[->] (5)--node[left]{\tiny$cx$}(6);
    \draw[->] (1)--node[left]{\tiny$za$}(6);
\end{tikzpicture}
\end{center}

Notice that the vertices $a$, $b$, $c$ must be mutually distinct since $Q$ is Schurian. Hence $ab(=ba), bc, ca$ must be mutually distinct also.

 If $\{a,b,c\}\cap\{ab,bc,ca\}=\emptyset$, i.e. all six vertices are mutually distinct, then there is a  $\tilde{A_5}$ type subquiver  of the $\Ext$-quiver $Q$. By Lemma \ref{non discrete type}, $H$ is not discrete corepresentation type, which is a contradiction.

Now assume $\{a,b,c\}\cap\{ab,bc,ca\}\neq\emptyset$. It is easy to see that six arrows $ay$, $xb$, $bz$, $yc$, $cx$ and $za$ in the diagram above are still mutually distinct.  Then the $\Ext$-quiver $Q$ contains a subquiver $Q'$, which is a quotient quiver of the above diagram.

\begin{figure}[h]
  \begin{center}
\begin{tikzpicture}[->]
\node(1) at (0,0) {$1$};
\node(2) at (0,1) {$a$};
\node(4) at (1,1) {$ab$};
\node(3) at (1,0) {$b$};
\node (c) at (-1,0.5) {$Q'$:};
\draw(1)--(2);
\draw(2)--(4);
\draw(3)--(4);
\draw(1)--(3);

\draw [->] (2.west)arc(330:30:0.15);
\draw [->] (3.east)arc(150:-150:0.15);
\end{tikzpicture}
\end{center}
\caption{The subquiver $Q'$ when $c=1$, $a\neq b^{-1}$}
\end{figure}
So by Lemma \ref{lem:quotient quiver}, $H$ is not discrete corepresentation type, which is a contradiction.
\end{proof}

\subsection{}
%So we reduce to the case with at most two arrows $x:1\rightarrow a$ and $y:1\rightarrow b$ out of $1$ exist. This translates to the the corresponding statement on skew primitives.

 Now we can make a brief observation about the possible $\Ext$-quivers for $\Mm^H$. Due to Proposition \ref{homogeneous}, the $\Ext$-quiver will be determined by the  L.I.C. $B\subseteq H$ containing $1$ (see Definition \ref{LIC}).

% \begin{remark}
% The sub-coalgebra $B$ is called a link-indecomposable component of $H$ in \cite{Mon}. Since link-indecomposable components of a pointed coalgebra $C$ are obtained from connected components of a quiver $\Gamma_C$ (\cite[Definition 1.1]{Mon}) which is slightly different from our $\Ext$-quiver of $C$ in the sense that it ignores multiple arrows between vertices, we decide not to use this terminology in order to avoid unnecessary confusion.
% \end{remark}

We are going to study the possibilities for the group $G(B)$ of group-likes of $B$ and the $\Ext$-quiver $Q$.

\begin{proposition}\label{finite case}
Assume $H$ is discrete corepresentation type and there are group-like elements $a$ and $b$ such that $P(1,a)$ and $P(1,b)$ contain non-trivial skew primitives.  If the subgroup $G(B)=\langle a,b\rangle$ of $G(H)$ is finite, then $a=b$. In this case, there is only $1$ arrow coming out/ going in to each vertex in the $\Ext$-quiver $Q$.  In particular, if $H$ is finite dimensional, then $a$=$b$.
\end{proposition}
\begin{proof}
If $a\neq b$, then in the $\Ext$-quiver, there are two arrows going out of each vertex. By Lemma \ref{ab=ba}, $ab=ba$.
%say there are arrows $x:1\to a$, $y:1\to b$. Then there are subquivers:
%$$
%\begin{tikzpicture}[->]
%\node (a) at (0,1) {$1$};
%\node (b) at (-1,0) {$a$};
%\node (c) at (1,0) {$b$, };
%
%\draw(a)--node[left]{$x$}(b);
%\draw(a)--node[right]{$y$}(c);
%
%\node (a2) at (3,1) {$a^{-1}b$};
%\node (b2) at (2,0) {$b$};
%\node (c2) at (4,0) {$ba^{-1}b$, };
%
%\draw(a2)--node[left]{$xa^{-1}b$}(b2);
%\draw(a2)--node[right]{$ya^{-1}b$}(c2);
%
%\node at (5,0.5) {$\cdots$};
%
%\node (a3) at (7,1) {$(a^{-1}b)^n$};
%\node (b3) at (6,0) {$b(a^{-1}b)^{n-1}$};
%\node (c3) at (8,0) {$b(a^{-1}b)^n$};
%
%\draw(a3)--node[left]{$ $}(b3);
%\draw(a3)--node[right]{$ $}(c3);
%
%\node at (9,0.5) {$\cdots$};
%
%\end{tikzpicture}
%$$
Since the subgroup generated by $a$ and $b$ is finite, assume $n>1$ is the smallest integer such that $(a^{-1}b)^n=1$. Then, there is a subquiver:
$$
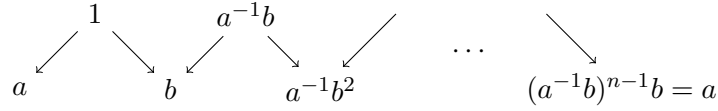
\begin{tikzpicture}[->]
\node (a) at (0,1) {$1$};
\node (b) at (-1,0) {$a$};
\node (c) at (1,0) {$b$};
\node(a2) at (2,1) {$a^{-1}b $};
\node (c2) at (3,0) {$a^{-1}b^2$};
\node(bn)at (7,0){$(a^{-1}b)^{n-1}b=a$};
\draw(a)--(b);
\draw(a)--(c);
\draw(a2)--(c);
\draw(a2)--(c2);
\draw(4,1)--(c2);
\node at(5,.5){$\cdots$};
\draw (6,1)--(bn);
\end{tikzpicture}
$$
of the $\Ext$-quiver of $\Mm^H$.
By Lemma \ref{non discrete type}, $H$ is not discrete type, which is a contradiction. The rest parts of the statements are clear.
\end{proof}

\begin{proposition}\label{2 arrows case}
Assume $H$ is discrete corepresentation type and there are group-like elements $a$ and $b$ such that $P(1,a)$ and $P(1,b)$ contain non-trivial skew primitives.  If $a\neq b$, then $G(B)=\langle a,b| ab=ba, a^m=b^n\rangle$ for some integers $(m,n)\neq \pm(1,1)$.
\end{proposition}
\begin{proof}
By the assumption and Proposition \ref{two arrows}, there are exactly two arrows coming out of $1$. So $G(B)$ is a group generated by $a$ and $b$. We have shown in Lemma \ref{ab=ba} that $ab=ba$. So $G(B)$ is also an abelian group, and hence a quotient of the free abelian group $\langle a,b\rangle$.  That is $G(B)=\langle a,b|ab=ba\rangle/G_0$ for some subgroup $G_0\subseteq \langle a,b|ab=ba\rangle$. Notice that $G_0$ is also free and $rank (G_0)\leq 2$. If $G_0$ has rank $2$, then $G(B)$ is a finite group. Hence by Theorem \ref{finite case} $a=b$, a contradiction. Therefore, $rank (G_0)\leq 1$ which means $G_0$ is the subgroup of $\langle a,b|ab=ba\rangle$ generated by $a^mb^{-n}$ for some $m, n$ ($rank (G_0)=0$ if and only if $m=n=0$). So $G(B)=\langle a,b| ab=ba, a^m=b^n\rangle$.
 \end{proof}

 %  In algebra terms, I think we can write $H$ as
%
%$$H=B\otimes_{\KK[G(B)]}G(H)$$
%	
%at least as algebras, and possibly as some Hopf algebras, given the fact that $G(B)$ is in the center of $H$. Discrete representation type of $H$ reduces to the discrete representation type of $B$, since as coalgebras, $H$ is a direct sum of subcoalgebras isomophic to $B$ (translations as above). The above can also be seen, I think, as the (Hopf algebra) tensor product $B\otimes \KK G(H)$ modulo the ideal generated by $a\otimes 1-1\otimes a$.

%We will have two cases: the first, when $G(B)$ is finite, we obtain that $B$ is Nakayama, and $H$ too, and it reduces to existing classifications.

We can now list all the possible $\Ext$-quivers. First we introduce the following notation.

\begin{definition}\label{defn:Qmn}
Let $(m,n)$ be a pair of integers such that $(m,n)\neq \pm(1,1)$. Let $Q^{m,n}$ be a quiver defined as the following: The set of vertices of $Q_0^{m,n}=\langle a,b| ab=ba, a^m=b^n\rangle$, for each vertex, there is a unique arrow $a^ib^j\to a^{i+1}b^{j}$ and a unique arrow $a^ib^j\to a^{i}b^{j+1}$. We call $Q^{0,0}$ the grid quiver and $Q^{m,n}$ a folded grid quiver when $(m,n)\neq (0,0)$.
\end{definition}

\begin{figure}[h]
$\begin{tikzpicture}[->]
\node(1) at (-1,1) {$b^2$};
\node(2) at (0,1) {$\cdot$};
\node(3) at (1,1) {$\cdot$};

\node(4) at (-1,0) {$b$};
\node(5) at (0,0) {$ab$};
\node(6) at (1,0) {$\cdot$};

\node(7) at (-1,-1) {$1$};
\node(8) at (0,-1) {$a$};
\node(9) at (1,-1) {$a^2$};

\draw (-2,1)node[left]{$\cdots$}--(1);
\draw (1)--(2);
\draw (2)--(3);
\draw (3)--(2,1)node[right]{$\cdots$};

\draw (-2,0)node[left]{$\cdots$}--(4);
\draw (4)--(5);
\draw (5)--(6);
\draw (6)--(2,0)node[right]{$\cdots$};

\draw (-2,-1)node[left]{$\cdots$}--(7);
\draw (7)--(8);
\draw (8)--(9);
\draw (9)--(2,-1)node[right]{$\cdots$};

\draw (-1,-2)node[below]{$\cdots$}--(7);
\draw (7)--(4);
\draw (4)--(1);
\draw (1)--(-1,2)node[above]{$\cdots$};

\draw (0,-2)node[below]{$\cdots$}--(8);
\draw (8)--(5);
\draw (5)--(2);
\draw (2)--(0,2)node[above]{$\cdots$};

\draw (1,-2)node[below]{$\cdots$}--(9);
\draw (9)--(6);
\draw (6)--(3);
\draw (3)--(1,2)node[above]{$\cdots$};

\end{tikzpicture}$

\caption{The grid quiver $Q^{0,0}$}
\end{figure}
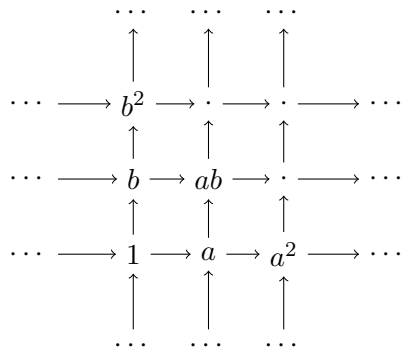

 \begin{theorem} \label{ext quiver}
If $H$ is a discrete corepresentation type Hopf algebra and $B$ is the maximal connected sub-coalgebra of $H$ containing $1$. Then the $\Ext$-quiver $Q$ of $\Mm^B$ is one of the following:

$(1)$ The single vertex $1$ and $G(B)=\{1\}$.

$(2)$ A complete oriented cycle, and $G(B)=\langle a | a^n=1\rangle$;

$(3)$ $\xymatrix{\cdots\ar[r]&\cdot\ar[r]&\cdot \ar[r]&\cdot\ar[r]&\cdots}$, and $G(B)=\langle a \rangle$;

$(4)$  $Q^{m,n}$ and $G(B)=\langle a,b| a^m=b^n\rangle$ for some integers $(m,n)\neq\pm(1,1)$.

The $\Ext$-quiver of $\Mm^H$ is a multiple (possibly infinite) copy of one of the above quiver.
%Furthermore, if $G(B)$ is finite, then $\Ext$-quiver of $B$ could only be case (1).
\end{theorem}

\begin{proof}
By Proposition \ref{two arrows}, there are at most two arrows coming out of the vertex $1$ in $Q$.

If there is no arrow coming out of $1$, then $G(B)=\{1\}$ and $Q$ is a single vertex.

If there is only one arrow coming out of $1$, then $G(B)$ is generated by only one group like element $a$. Hence  $Q$ is either (2) or (3).

If there are two arrows coming out of $1$,  then by Proposition \ref{2 arrows case}, $G(B)=\langle a,b| a^m=b^n\rangle$ for some integers $(m,n)\neq\pm(1,1)$ and hence $Q=Q^{m,n}$.

The $\Ext$-quiver of $\Mm^H$ follows from Lemma \ref{B tensor G}.
\end{proof}
%
%\begin{corollary}
%If $H$ is a discrete corepresentation type Hopf algebra and $B$ is the maximal connected sub-coalgebra of $H$ containing $1$. Then $G(B)$ is isomorphic to one of the following (1) $\mathbb Z$, (2) $\mathbb Z/n \mathbb Z$, (3) $\mathbb Z\times \mathbb Z$, (4) $\mathbb Z\times \mathbb Z/n \mathbb Z$.
%\end{corollary}

\begin{remark}
In fact in case (4), to ensure $H$ is a discrete corepresentation type Hopf algebra, we need more restriction on the pair of intergers $(m,n)$ (see Lemma \ref{m+n} below).
\end{remark}

Next we need to determine the algebra structure of the Hopf algebras $B$ and $H$ according to these four cases. 
%Recall that a left (right) comodule $M$ is called {\bf uniserial} if its lattice of subcomodules is a chain. A coalgebra is {\bf left (right) serial} if it is a direct sum of left (right) uniserial comodules. A coalgebra is called {\bf serial} if it is left and right serial. A Hopf algebra $H$ is called {\bf coserial} if it is serial as a coalgebra.  
When the $\Ext$-quiver $Q$ is case (1)--(3), from \cite[Theorem 2.10]{CGT} $H$ is a coserial Hopf algebra (see Definition \ref{coserial defn}).  Conversely, every pointed coserial Hopf algebra must be discrete corepresentation type. Such class of Hopf algebras has been classified in \cite{IJ}. In the following section, we will focus on the remaining case (4).

\section{The Hopf algebra structures}

% Here we bring in Shijie's observation that $H\neq H_1$, i.e. in $H$ there are some elements in $\KK Q$ which are not spanned by grouplikes and arrows. (I think also by a result in \cite{DIN} probably elsewhere, the span of arrows in the path coalgebra cannot a Hopf algebra unless the quiver has no paths of length more than 2 or something). Recall that this is done by looking at $xy$; I wonder if we might need something about characteristic here.

Throughout this section, let $\KK$ be an algebraic closed field with $char \KK\neq 2$.
Let $H$ be a Hopf algebra of discrete corepresentation type and $B$  the L.I.C. of $H$ containing $1$. Denote by $Q_B$ the $\Ext$-quiver of $\Mm^B$ which equals $Q^{m,n}$  for some $m,n$ and denote by $Q_H$ the $\Ext$-quiver of $\Mm^H$.

 We identify arrows $1\to a$ and $1\to b$ in $Q_H$ with chosen non-trivial skew primitive elements  $x\in P(1,a)$ and $y\in P(1,b)$ respectively. By translation, the arrow $a^ib^j\to a^{i+1}b^j$ is identified with $a^ib^j x$ and  the arrow $a^ib^j\to a^{i}b^{j+1}$ is identified with $a^ib^j y$. It is clear that $a^ib^jx-rxa^ib^j\in H_0\cap P(a^ib^j, a^{i+1}b^j)$ for some number $r\neq 0$, since there is a unique arrow $a^ib^j\to a^{i+1}b^j$. The same statement holds for $y$ as well. Use the notation $(-|-)$ for the concatenation of arrows, e.g. $(x|ay)$ is the path $1\to a\to ab$.

 Since $Q^{m,n}$ is the $\Ext$-quiver of $\Mm^B$, we have $\KK Q^{m,n}_1\subseteq B\subseteq \KK Q^{m,n}$ as coalgebras, where $\KK Q^{m,n}_1=\KK \{a^ib^j, a^ib^jx, a^ib^jy\}$ is the coalgebra in the coradical filtration of $\KK Q^{m,n}$.

 In order to  simultaneously discuss different cases for the quiver $Q^{m,n}$ depending on various choices of the integers $(m,n)$, we introduce the notation $[[\cdots]]$ for a multiset and $\Set$ is the forgetful functor from the category of multiset to the category of sets. For example, $\Set[[1, 1, 2, 2, 2, 3]] = \{1, 2, 3\}$.

 First, we mention the translation technique which will be heavily used later:
 \begin{lemma}\label{translation}
 Let a bialgebra $C$ be a subcoalgebra of a path coalgebra $\KK Q$ such that $\KK Q_1\subseteq C$. If a path $p=(x_1|x_2)\in C$, then $(gx_1|gx_2)\in C$ for any group like element $g\in C$.
 \end{lemma}
 \begin{proof}
 Suppose the path $p$ starts from $s$ and ends at $t$.
 Since $$\Delta_{\KK Q}(g(x_1|x_2)-(gx_1|gx_2))=gs\otimes (g(x_1|x_2)-(gx_1|gx_2))+(g(x_1|x_2)-(gx_1|gx_2))\otimes gt$$  So $g(x_1|x_2)-(gx_1|gx_2)\in P(gs,gt)\subseteq C$. Because  $g(x_1|x_2)\in C$, it follows that $(gx_1|gx_2)\in C$.
 \end{proof}

\begin{lemma}\label{xy1}\
\begin{enumerate}
\item  $H$ contains a non-zero linear combination of paths $(x|ay)$ and $(y|bx)$.
\item $H$ contains neither the path $(x|ay)$ nor $(y|bx)$.
\end{enumerate}
\end{lemma}
\begin{proof}
(1)
Assume $xb=\lambda bx+\mu(b-ab)$ for some $\lambda\neq 0, \mu\in \KK$. Since
$$
\Delta_{\KK Q_H} (x|ay)=1\otimes (x|ay)+x\otimes ay +(x|ay)\otimes ab
$$
$$
\Delta_{\KK Q_H} (y|bx)=1\otimes (y|bx)+y\otimes bx +(y|bx)\otimes ab
$$

One can check that
$$
\Delta((x|ay)+\lambda(y|bx)-xy+\mu y)=1\otimes ((x|ay)+\lambda(y|bx)-xy+\mu y)+((x|ay)+\lambda(y|bx)-xy+\mu y)\otimes ab
$$

which implies $(x|ay)+\lambda(y|bx)-xy+\mu y\in P(1,ab)\subseteq H$.
Hence $(x|ay)+\lambda(y|bx)\in H$.

%With a similar argument using the assumption $ya=r'ay+s'(a-ab)$ for some $0\neq r',s'\in \KK$, $r'(x|ay)+(y|bx)\in H$.

(2) By (1), if either $(x|ay)\in H$ or $(y|bx)\in H$, then both of them are in $H$. Let $D$ be the subcoalgebra of $H$ linearly generated by $\Set[[1,a,b,ab,x,y,bx,ay,(x|ay),(y|bx)]]$.  We want to show that it is infinite representation type.

 \begin{figure}[h]
 $$
 \begin{tikzpicture}[->]
 \node(-0) at (-4,0) {$1$};
\node(-1) at (-3,0) {$a$};
\node(-2) at (-4,1) {$b$};
\node(-3) at (-3,1) {$ab$};
\draw(-0)--node[below]{$x$}(-1);
\draw(-0)--node[left]{$y$}(-2);
\draw(-2)--node[above]{${bx}$}(-3);
\draw(-1)--node[right]{${ay}$}(-3);
\node at (-3.5,-1) {$Q_B=Q^{0,0}$};

 \node(0) at (0,0) {$1$};
  \node(1) at (0,1) {$b$};

\draw(0)--node[right]{$y$}(1);
\draw (0.west)arc(330:30:0.15)node[left]{$x$\ \ \ };
\draw (1.west)arc(330:30:0.15)node[left]{$bx$\ \ \ };;
\node at (0,-1) {$Q_B=Q^{1,0}$};

\node(11) at (2,0) {$b$};
\node (22) at (3,0) {$1$};
\node (33) at (4,0) {$a$};

\draw(11) to [out=45,in=135]node[above]{\tiny $bx$}(22);
\draw(22) to [out=45,in=135]node[above]{\tiny $x$}(33);
\draw(33) to [out=-135,in=-45]node[below]{\tiny $ay$}(22);
\draw(22) to [out=-135,in=-45]node[below]{\tiny $y$}(11);
\node at (3,-1)  {$Q_B=Q^{1,-1}$};
 \end{tikzpicture}
 $$
 \caption{Possible $\Ext$-quiver of $\Mm^D$}
 \end{figure}

Notice that $\Set[[1,a,b,ab,x,y,bx,ay,(x|ay),(y|bx)]]$ is a diamond basis of $D$.
Consider the path coalgebra $\tilde D=\KK\tilde{Q}$:

$$\begin{tikzpicture}[->]
\node at (-1,1) {$\tilde Q:$};
\node(0) at (0,0) {$\tilde 1$};
\node(1) at (1,0) {$\tilde a$};
\node(2) at (0,1) {$\tilde b$};
\node(3) at (1,1) {$\widetilde{ab}$};
\draw(0)--node[below]{$\tilde x$}(1);
\draw(0)--node[left]{$\tilde y$}(2);
\draw(2)--node[above]{$\widetilde {bx}$}(3);
\draw(1)--node[right]{$\widetilde {ay}$}(3);
 \end{tikzpicture}
 $$
where paths of $\tilde D$ form a diamond basis of $\tilde D$.

Consider the coalgebra map: $\pi:\tilde D\to D$ defined by
$$
\pi(\tilde 1)=1, \pi(\tilde a)=a, \pi(\tilde b)=b, \pi(\widetilde ab)=ab,
$$
$$
 \pi(\tilde x)=x, \pi(\tilde y)= y,  \pi(\widetilde {bx})=bx, \pi(\widetilde {ay})=ay,
 $$
$$
\pi((\tilde x|\widetilde {ay}))= (x|ay), \pi((\tilde y|\widetilde {bx}))= (y|bx).
$$

It is easy to see that $\pi$ is a covering map. Since $\tilde D$ is infinite representation type, by Lemma \ref{covering fin type}, so is $D$. But then $H$ contains a finite dimensional infinite representation type  sub-coalgebra, contradicting with the fact that $H$ is discrete corepresentation type.
 \end{proof}

By translation, same statements hold for all paths $(a^ib^jx|a^{i+1}b^jy)$ and $(a^ib^jy|a^ib^{j+1}x)$.

\begin{lemma} \label{x square}\
 $H$ contains neither the path $(x|ax)$ nor $(y|by)$.

\end{lemma}
\begin{proof}
The $\Ext$-quiver $Q_B=Q^{m,n}$ contains a subquiver $Q_C$ supported on the vertices in Figure \ref{yy}.
\begin{figure}[h]
$$\begin{tikzpicture}[->]
\node(0) at (-1.5,3) {\tiny $a^{-1}b^2$};
\node(1) at (-1.5,2) {\tiny $a^{-1}b$};
\node(2) at (0,2) {$b$};

\node(4) at (0,3) {$b^2$};
\node(5) at (0,1) {$1$};
\node(6) at (1,1) {$a$};

\node(7) at (1,0) {\tiny $ab^{-1}$};
\node(8) at (2.5,0) {\tiny $a^2b^{-1}$};
\node(9) at (1,-1) {\tiny $ab^{-2}$};
\node(10) at (2.5,-1) {\tiny $a^2b^{-2}$};

\draw(0)--(4); \draw(1)--(0); \draw(10)--(8); \draw(9)--(10);
\draw(1)--(2); \draw(2)--node[left]{\tiny$by$}(4);\draw(5)--node[left]{\tiny$y$}(2);\draw(5)--node[above]{\tiny$x$}(6);
\draw(7)--(6); \draw(7)--(8); \draw(9)--(7);
\end{tikzpicture}$$
\caption{} \label{yy}
\end{figure}

If $H$ contains the path $(y|by)$, and hence all paths $(a^ib^jy|a^ib^{j+1}y)$, then
$$C=\Span_\KK\{(Q_C)_0\cup ({Q_C})_1\cup \{(y|by), (ab^{-2}y|ab^{-1}y), $$$$(a^{-1}bx|by)+\lambda(a^{-1}by|a^{-1}b^2x), (ab^{-2}x|a^2b^{-2}y)+\lambda(ab^{-2}y|ab^{-1}x) \}\}$$
is a finite dimensional sub-coalgebra of $H$.

Furthermore, $C$ has a diamond basis $$\Set((Q_C)_0\cup ({Q_C})_1\cup [[(y|by), (ab^{-2}y|ab^{-1}y), (a^{-1}bx|by)+\lambda(a^{-1}by|a^{-1}b^2x), (ab^{-2}x|a^2b^{-2}y)+\lambda(ab^{-2}y|ab^{-1}x) ]]).$$
We need to show that $C$ is infinite representation type.

Consider the coalgebra $\tilde C\subseteq\KK\tilde Q$  for the following quiver $\tilde Q$,
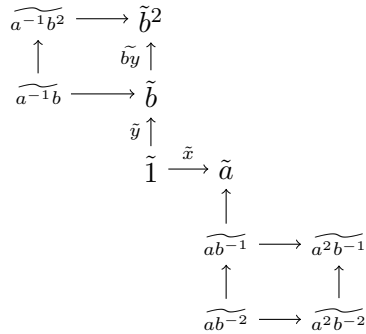
$$\begin{tikzpicture}[->]
\node(0) at (-1.5,3) {\tiny $\widetilde{a^{-1}b^2}$};
\node(1) at (-1.5,2) {\tiny $\widetilde{a^{-1}b}$};
\node(2) at (0,2) {$\tilde b$};

\node(4) at (0,3) {$\tilde b^2$};
\node(5) at (0,1) {$\tilde 1$};
\node(6) at (1,1) {$\tilde a$};
%\node(66) at (1,0) {$\cdot$};

\node(7) at (1,0) {\tiny $\widetilde{ab^{-1}}$};
\node(8) at (2.5,0) {\tiny $\widetilde{a^2b^{-1}}$};
\node(9) at (1,-1) {\tiny $\widetilde{ab^{-2}}$};
\node(10) at (2.5,-1) {\tiny $\widetilde{a^2b^{-2}}$};

\draw(0)--(4); \draw(1)--(0); \draw(10)--(8); \draw(9)--(10);
\draw(1)--(2); \draw(2)--node[left]{\tiny$\widetilde{by}$}(4);\draw(5)--node[left]{\tiny$\tilde y$}(2);\draw(5)--node[above]{\tiny$\tilde x$}(6);
\draw(7)--(6); \draw(7)--(8); \draw(9)--(7);
\end{tikzpicture}$$

where $\tilde C$ is linearly generated by the diamond basis $$\tilde D=\tilde Q_0\cup \tilde Q_1\cup\{(\tilde y|\widetilde{by}), (\widetilde{ab^{-2}y}|\widetilde{ab^{-1}y}), (\widetilde{a^{-1}bx}|\widetilde{by})+\lambda(\widetilde{a^{-1}by}|\widetilde{a^{-1}b^2x}), (\widetilde{ab^{-2}x}|\widetilde{a^2b^{-2}y})+\lambda(\widetilde{ab^{-2}y}|\widetilde{ab^{-1}x}) \}.$$

Then we claim that the coalgebra homomorphism $\pi:\tilde C\to C$ defined by
$$\pi(\widetilde{a^ib^j})=a^ib^j, \pi(\widetilde{a^ib^jx})=a^ib^jx, \pi(\widetilde{a^ib^jy})=a^ib^jy,$$
$$\pi(\tilde y|\widetilde{by})=(y|by), \pi(\widetilde{ab^{-2}y}|\widetilde{ab^{-1}y})=({ab^{-2}y}|{ab^{-1}y})$$,
$$\pi( (\widetilde{a^{-1}bx}|\widetilde{by})+\lambda(\widetilde{a^{-1}by}|\widetilde{a^{-1}b^2x}))=(a^{-1}bx|by)+\lambda(a^{-1}by|a^{-1}b^2x),$$ $$\pi((\widetilde{ab^{-2}x}|\widetilde{a^2b^{-2}y})+\lambda(\widetilde{ab^{-2}y}|\widetilde{ab^{-1}x}) )=(ab^{-2}x|a^2b^{-2}y)+\lambda(ab^{-2}y|ab^{-1}x) $$
is a covering map.

In fact, in $\tilde D$ the only four pairs of diamonds with same sources and different targets: (1) $\tilde x$ and $\tilde y$; (2) $\widetilde{a^{-1}bx}$ and $\widetilde{a^{-1}by}$; (3)  $\widetilde{ab^{-1}x}$ and $\widetilde{ab^{-1}y}$;
(4)  $\widetilde{ab^{-2}x}$ and $\widetilde{ab^{-2}y}$. If any pair is sent to the same diamond under the map $\pi$, then $a=b$ which is impossible.
Similarly, any pair of diamonds with the same targets but different sources cannot be mapped to the same diamond. Hence $\pi$ is a covering map.

Since $\tilde C$ is a finite dimensional coalgebra and $\tilde C^*$ is infinite representation type as we have shown in Example \ref{localization exm}, $\tilde C$ is infinite representation type. By Lemma \ref{covering fin type}, $C$ is infinite representation type. Hence $H$ contains a finite dimensional infinite representation type sub-coalgebra, contradicting with the fact that $H$ is discrete corepresentation type.

Therefore $(y|by)\not\in H$ and similarly $(x|ax)\not\in H$.
 \end{proof}

\begin{lemma} \label{xa1}
We can choose $x$ and $y$ such that
$$
ax+xa=0, by+yb=0,
 x^2\in P(1,a^2), y^2\in P(1,b^2),
$$
\end{lemma}
\begin{proof}
 Assume  $xa=\lambda ax+\mu(a-a^2)$ for some number $0\neq \lambda, \mu\in \KK$.  Since
$$\Delta(x^2) = (1\otimes x+x\otimes a)(1\otimes x+x\otimes a)= 1\otimes x^2+x\otimes xa+ x\otimes ax+ x^2\otimes a^2$$
$$\Delta_{\KK Q_H}(x|ax)=1\otimes (x|ax)+x\otimes ax+(x|ax)\otimes a^2,$$
one can check $$\Delta(x^2-(\lambda+1)(x|ax)-\mu x)=1\otimes (x^2-(\lambda+1)(x|ax)-\mu x)+(x^2-(\lambda+1)(x|ax)-\mu x)\otimes a^2.$$
 So $x^2-(\lambda+1)(x|ax)-\mu x\in P(1,a^2)$, which implies $(\lambda+1)(x|ax)\in H$. But the fact that $(x|ax)\not\in H$ forces $\lambda=-1$. Hence we have $ax+xa=\mu(a-a^2)$.  After a modification $x'=x-\frac{\mu}{2}(1-a)$ (here we use the assumption $char F\neq 2$), it follows that $ax'+x'a=0$. In this case, ${x'}^2\in P(1, a^2)$.
Similarly, we can choose $y$, such that $by+yb=0$ and $y\in P(1,b^2)$.
\end{proof}

In the following, always assume we make the above choice of $x$ and $y$.

\begin{lemma}\label{xb1}
There are $\lambda_x$, $\lambda_y\in\KK^\times$, such that $bx+\lambda_x xb=0$ and $ay+\lambda_y ya=0$.
\end{lemma}
\begin{proof}
Notice that $b^{-1}xb\in P(1,a)$ is a non-trivial skew primitive, $b^{-1}xb=-\lambda_x x+\mu_x(1-a)$ for some $\lambda_x\neq 0$. If $a=1$, we are done. Otherwise,
since $ab=ba$ and $ax+xa=0$, $ab^{-1}xb+b^{-1}xba=0$. The left hand side is
$$
ab^{-1}xb+b^{-1}xba=\lambda_x ax+\mu_x(a-a^2)+\lambda_xxa+\mu_x(a-a^2)=2\mu_x(a-a^2).
$$
Hence $\mu_x=0$ and $bx+\lambda_x xb=0$.  (use $char \KK\neq 2$ again.)

Similarly $ay+\lambda_y ya=0$.
\end{proof}

\begin{lemma}\label{xy2} \
\begin{enumerate}
\item There is $\lambda\in\KK^\times$, such that $\lambda_x=\lambda^{-1}$ and $\lambda_y=\lambda$.
\item   $xy+\lambda yx\in P(1,ab)$.
\end{enumerate}
\end{lemma}
\begin{proof}
Since $bx+\lambda_x xb=0$ and $ay+\lambda_y ya=0$ by Lamma \ref{xb1},

$$
\Delta(xy) =1\otimes xy+y\otimes xb+x\otimes ay+xy\otimes ab=1\otimes xy-\frac{1}{\lambda_x}y\otimes bx+x\otimes ay+xy\otimes ab,
$$
$$
\Delta(yx)  =1\otimes yx+x\otimes ya+y\otimes bx+yx\otimes ab=1\otimes yx-\frac{1}{\lambda_y}x\otimes ay+y\otimes bx+yx\otimes ab.
$$

Set $p=xy+\lambda_yyx-(\lambda_y-\frac{1}{\lambda_x})(y|bx)$, one can check that
$
\Delta_{\KK Q_H} (p)=1\otimes p+p\otimes ab.
$

So $p=xy+\lambda_yyx-(\lambda_y-\frac{1}{\lambda_x})(y|bx)\in P(1,ab)$. i.e.
$$
(\lambda_y-\frac{1}{\lambda_x})(y|bx)=xy+\lambda_yyx+ P(1,ab)\in H.
$$
But $(y|bx)\not\in H$ by Lemma \ref{xy1}. It follows that $\lambda_y=\frac{1}{\lambda_x}$ and $xy+\lambda_yyx\in P(1,ab)$. Set $\lambda_y=\lambda$, both statements hold immediately.

\end{proof}

\begin{lemma}\label{m+n}
If $a^n=b^m$, then $n+m\in 2\mathbb Z$.
\end{lemma}
\begin{proof}
Since $b^{-1}xb=-\lambda x$, it follows that $(-1)^nx=a^{-n}xa^n=b^{-m}xb^m=(-\lambda)^m x$. Hence $\lambda^m=(-1)^{m+n}$.
Similarly, from the condition $a^{-1}ya=-\lambda^{-1} y$, it following that $\lambda^n=(-1)^{m+n}$. Therefore,
$$
\lambda^n=\lambda^m=(-1)^{m+n}.
$$

Suppose $m+n$ is and odd number. Without loss of generality, say $m$ is odd and $n$ is even. Let $d=gcd(m,n)$, which is an odd number. From the above equalities, we know that $\lambda^d=-1$. But then $\lambda^n=1$ which is a contradiction.
\end{proof}

\begin{corollary}
If $(m,n)\neq (0,0)$, then $\lambda$ is a $d$-th root of unity, where $d=gcd(m,n)$.
\end{corollary}
\begin{remark}
$gcd(m,0)=m$.
\end{remark}

\begin{corollary}
 There are $s,t,k\in\KK$ such that $$x^2=s(1-a^2), y^2=t(1-b^2), xy+\lambda yx=k(1-ab).$$
\end{corollary}
\begin{proof}
From Lemma \ref{xa1}, $x^2\in P(1,a^2)$. Since $a^2\neq a$ or $b$ due to Lemma \ref{m+n}, there is no arrow in $Q^{m,n}$ from $1$ to $a^2$. Thus $x^2=s(1-a^2)$ is a trivial primitive. The other two equality hold similarly.
\end{proof}

We now construct a family of Hopf algebras which reflect the above restrictions. Let $(m,n)$ be a pair of integers with $(m,n) \neq \pm (1,1)$ and $m+n \in 2\mathbb{Z}$. Let $\lambda \in \KK^{\times}$, with $\lambda$ a $\operatorname{gcd}(m,n)^{th}$ root of unity whenever $(m,n) \neq (0,0)$. Let $G^{m,n}$ denote the abelian group $\langle a,b \rangle /(ab=ba, a^m = b^n)$, and $R = \KK G^{m,n}$ the group algebra. Then $X := \KK Q_1^{m,n}$ can be considered a left $R$-module via translation of arrows. In fact, it can be given the structure of a Hopf bimodule: the right $R$-module structure is induced from the equations
\begin{eqnarray*}
 xa = -ax, xb = -\lambda bx, \\
yb = -by, ya = -\frac{1}{\lambda}ay.\\
\end{eqnarray*}
The left and right comodule structures are given by
\begin{eqnarray*}
\rho_L(a^ib^jx ) = a^ib^j\otimes x, \rho_L(a^ib^jy) = a^ib^j\otimes y, \\
\rho_R(a^ib^jx) = x\otimes a^{i+1}b^j, \rho_R(a^ib^jy) = y\otimes a^ib^{j+1}.\\
\end{eqnarray*}
The tensor product $X\otimes_RX$ then carries the structure of a Hopf bimodule over $R$. Set $T^{m,n}_2 := R \oplus X \oplus (X\otimes_RX)$, which we denote by $T$ for brevity. Then $T$ can be considered a graded $k$-algebra: $T = \bigoplus_{i=0}^{\infty}{T_i}$ with $T_0 = R$, $T_1 = X$, $T_2 = X\otimes_RX$, and $T_i = 0$ for $i>2$. The multiplication is induced by the bimodule structures on the $T_i$'s, the natural map $X\otimes_\KK X \rightarrow X\otimes_RX$, and the requirement $T_iT_j \subset T_{i+j}$ for all $i,j \geq 0$. There is a unique coassociative algebra map $\Delta : T\rightarrow T\otimes_\KK T$ satisfying $\Delta(a^ib^j) = a^ib^j\otimes a^ib^j$, $\Delta (a^ib^j x) = a^ib^j\otimes x + x\otimes a^{i+1}b^j$, and $\Delta(a^ib^jy) = a^ib^j\otimes y + y \otimes a^ib^{j+1}$ for all integers $i$ and $j$. In turn, there is a unique algebra map $\epsilon : T \rightarrow \KK$ satisfying $\epsilon(a^ib^j) = 1$ and $\epsilon(T_1) = 0$. This endows $T$ with the structure of a bialgebra: it is then straightforward to verify that $T$ has a unique antipode compatible with this structure. \\
The multiplication on $T$ ensures that $xy+\lambda yx \in P(1,ab)$, $x^2 \in P(1,a^2)$ and $y^2 \in P(1,b^2)$. Therefore, the ideal $I = I^{m,n}(\lambda,s,t,k)$ generated by the relations
\begin{eqnarray*}
xy+\lambda yx = k(1-ab), x^2 = s(1-a^2), y^2 = t(1-b^2) \\
\end{eqnarray*}
is in fact a Hopf ideal. Therefore, $T/I = T^{m,n}_2 / I^{m,n}(\lambda, s,t,k)$ is a Hopf algebra. We summarize this information in the following definition.

\begin{definition}\label{defn Bmn}
For pairs of integers $(m,n)\neq\pm(1,1)$ and $m+n\in2\mathbb Z$, define $B^{m,n}(\lambda, s, t, k) = T^{m,n}_2/I^{m,n}(\lambda,s,t,k)$ to be the Hopf algebra constructed above. More explicitly, it is the Hopf algebra generated by $a,b,x,y$ satisfying the following conditions, where $\lambda\neq 0,s,t,k\in \KK$.
\begin{eqnarray*}
&&ab=ba, a^m=b^n, xy+\lambda yx=k(1-ab),\\
&&ax+xa=0, \lambda bx+xb=0, x^2=s(1-a^2),\\
&&by+yb=0, ay+\lambda ya=0, y^2=t(1-b^2);\\
&&\Delta(a)=a\otimes a, \Delta (b)=b\otimes b,
\Delta (x)=1\otimes x+x\otimes a, \Delta (y)=1\otimes y+y\otimes b;\\
&&\epsilon(a)=\epsilon(b)=1, \epsilon(x)=(y)=0;\\
&&S(a)=a^{-1}, S(b)=b^{-1}, S(x)=-xa^{-1}, S(y)=-yb^{-1}.
\end{eqnarray*}
\end{definition}

From the definition, it follows automatically that $\lambda$ is a $gcd(m,n)$-th root of unity when $(m,n)\neq (0,0)$.
 We want to discuss when $B^{m,n}(\lambda,s,t,k)$ is discrete corepresentation type. First, we need the following preparation:
\begin{proposition}\label{basis}
The Hopf algebra $B^{m,n}(\lambda,s,t,k)$ has a basis
\begin{enumerate}
\item $\{a^ib^jx^py^q | i,j\in\mathbb Z, p,q=0,1\}$ when $m=n=0$;
\item  $\{a^ib^jx^py^q | 0\leq i<m, j\in\mathbb Z, p,q=0,1\}$ when $m>0$;
\item $\{a^ib^jx^py^q | i,j\in\mathbb Z, 0\leq j<n, p,q=0,1\}$, when $n>0$.
\end{enumerate}
\end{proposition}
\begin{proof}
We prove (1). The other two cases are similar.

It is clear that each element of $B^{0,0}(\lambda,s,t,k)$ can be written as a linear combination of $a^ib^jx^my^n$, i.e. a linear combination of $a^ib^j$, $a^ib^jx$, $a^ib^jy$, $a^ib^jxy$. It suffice to show that these are linearly independent.

Suppose $\sum \alpha_{ij}a^ib^j+\beta_{ij} a^ib^jx+\gamma_{ij} a^ib^jy+\delta_{ij} a^ib^jxy=0.$
Since $a^ib^j\in B^{0,0}(\lambda,s,t,k)_0$, $a^ib^jx$, $a^ib^jy\in B^{0,0}(\lambda,s,t,k)_1\setminus B^{0,0}(\lambda,s,t,k)_0$, $a^ib^jxy\in B^{0,0}(\lambda,s,t,k)_2\setminus B^{0,0}(\lambda,s,t,k)_1$ it follows that $\sum \alpha_{ij}a^ib^j=0$, $\sum\beta_{ij} a^ib^jx+\gamma_{ij} a^ib^jy=0$ and  $\sum\delta_{ij} a^ib^jxy=0.$ It is clear that $\alpha_{ij}=\beta_{ij}=\gamma_{ij}=0$, because $a^ib^j$, $a^ib^jx$, $a^ib^jy$ represents different vertices and arrows. Finally, since $\sum\delta_{ij} a^ib^jxy=0$,
\begin{eqnarray*}
0&=&\sum\delta_{ij} \Delta(a^ib^jxy)\\
&=& \sum \delta_{ij} a^ib^j\otimes a^ib^j xy-\lambda  \delta_{ij} a^ib^j y\otimes a^ib^{j+1}x+ \delta_{ij} a^ib^j x\otimes a^{i+1}b^jy+ \delta_{ij} a^ib^j xy\otimes a^{i+1}b^{j+1}.
\end{eqnarray*}
Hence $0= \sum \delta_{ij} a^ib^j\otimes a^ib^j xy\in B^{0,0}(\lambda,s,t,k)_0\otimes B^{0,0}(\lambda,s,t,k)_2$. Therefore $\delta_{ij}=0$.
\end{proof}

\begin{corollary}
 $B^{m,n}(\lambda,s,t,k)= B^{m,n}(\lambda,s,t,k)_2$.
\end{corollary}

From the definition of the Hopf algebra $B^{m,n}(\lambda,s,t,k)$, $Q^{m,n}$ is the $\Ext$-quiver of the comodule category.
\begin{lemma} \label{B paths}
The Hopf algebra $B^{m,n}(\lambda,s,t,k)$ is a subcoalgebra of $\KK Q^{m,n}$.
\begin{enumerate}
\item  $B^{m,n}(\lambda,s,t,k)$ contains neither paths $(a^ib^jx|a^{i+1}b^jx)$ nor $(a^ib^jy|a^ib^{j+1}y)$.
\item  The linear combination $c_1(a^ib^jx|a^{i+1}b^jy)+c_2(a^ib^jy|a^ib^{j+1}x)\in B^{m,n}(\lambda,s,t,k)$ if and only if $c_2+\lambda c_1=0$.
\end{enumerate}
\end{lemma}
\begin{proof}
Without loss of generality we prove both statements for the case $i=j=0$.

(1) Notice that $\Delta(x|ax)=1\otimes (x|ax)+x\otimes ax+(x|ax)\otimes a^2$, where the middle term is a tensor product of two arrows of the same direction. However, the comultiplications of none of the basis elements contain such a term. Therefore $(x|ax)\not\in B^{m,n}(\lambda,s,t,k)$. Similarly, $(y|by)\not\in B^{m,n}(\lambda,s,t,k)$.

(2) Denote by $p=(x|ay)-\lambda(y|bx)$. Since $\Delta (p-xy)=1\otimes(p-xy)+(p-xy)\otimes ab$, $p-xy\in \KK Q^{m,n}_0\subseteq B^{m,n}(\lambda,s,t,k)$. So $p\in B^{m,n}(\lambda,s,t,k)$. Therefore any linear combination  $c_1(x|ay)+c_2(y|bx)\in B^{m,n}(\lambda,s,t,k)$, whenever $c_2+\lambda c_1=0$.

Conversely, define a $\KK$-linear map $\varphi:B^{m,n}(\lambda,s,t,k)\otimes B^{m,n}(\lambda,s,t,k)\to \KK$ by $\varphi(x\otimes ay)=\lambda$, $\varphi(y\otimes bx)=1$ and $\varphi=0$ on the other basis element. It is easy to check that the linear map $\varphi\circ\Delta: B^{m,n}(\lambda,s,t,k)\to\KK$ is trivial.
If  $c_1(x|ay)+c_2(y|bx)\in B^{m,n}(\lambda,s,t,k)$,   $\varphi(\Delta(c_1(x|ay)+c_2(y|bx)) )=c_1\lambda+c_2=0$.
\end{proof}

For an integer $N\geq 0$, Denote by  $T_N=\Span_\KK\{a^ib^jxy, a^{i+1}b^jy, a^ib^{j+1}x,a^{i+1}b^{j+1}| -N\leq i,j\leq N  \}$ a $B^{m,n}(\lambda,s,t,k)$ comodule and  $B^{m,n}_N(\lambda,s,t,k)=cf(T_N)$ the corresponding finite dimensional coefficient coalgebra.

\begin{theorem}\label{B is discrete}
The Hopf algebra $B^{m,n}(\lambda,s,t,k)$ is discrete corepresentation type if and only if $m\neq n$ or $m=n=0$.
\end{theorem}
 \begin{proof}
 When $m=n\neq 0$, $B^{n,n}(\lambda,s,t,k)\cmodd$ contains an infinite family of band modules:
 $$
\begin{tikzpicture}[->]

\node(0) at (0,1) {$\cdot$};
\node(1) at (0,0) {$\cdot$};
\node(2) at (1,0) {$\cdot$};
\node(3) at (1,-1) {$\cdot$};
\node(4) at (2,-1)  {$\cdot$};
 \node (5) at (2,-2){$a^n$};

 \node(a) at (-1,1) {$b^n$};
 \node (b) at (1,-0.5) {$\vdots$};

 \draw (1)--(0); \draw(1)--(2);
  \draw(3)--(4); \draw (5)--(4);
  \draw(a)--(0); %\draw(5)--(b);
\end{tikzpicture}
$$

 Hence $B^{n,n}(\lambda,s,t,k)$ is not discrete corepresentation type.

 Now assume $m\neq n$ or $m=n=0$.

Since any finite dimensional sub-coalgebra of $B^{m,n}(\lambda,s,t,k)$ is a sub-coalgebra of $B^{m,n}_N(\lambda,s,t,k)$ for some $N$. It suffice to show that $B^{m,n}_N(\lambda,s,t,k)$ is finite representation type.

In fact, since $B^{m,n}(\lambda,s,t,k)=B^{m,n}(\lambda,s,t,k)_2$,  $B^{m,n}_N(\lambda,s,t,k)$ only contains comodules with Loewy length at most $3$.

Indecomposable comodules with Loewy length $1$ are just simples $S=\Span_\KK\{a^ib^j\}$.

Due to Lemma \ref{B paths} (1),  indecomposable comodules with Loewy length $2$ are the strings $$C=\Span_\KK\{\cdots, a^ib^jy, a^ib^j, a^ib^j x, a^{i+1}b^{j-1}y, a^{i+1}b^{j-1}\cdots \}:$$
$$
\begin{tikzpicture}[->]

\node(0) at (0,1) {$\cdot$};
\node(1) at (0,0) {$\cdot$};
\node(2) at (1,0) {$\cdot$};
\node(3) at (1,-1) {$\cdot$};
\node(4) at (2,-1)  {$\cdot$};
 \node (5) at (2,-2){$\cdot$};

 \node(a) at (-1,1) {$\cdots$};
 \node (b) at (3,-2) {$\cdots$};

 \draw (1)--(0); \draw(1)--(2);
  \draw (3)--(2); \draw(3)--(4); \draw (5)--(4);
  \draw(a)--(0); \draw(5)--(b);
\end{tikzpicture}
$$
 No string can form a band, when $m\neq n$ or $m=n=0$.

Due to Lemma \ref{B paths} (2), indecomposable comodules with Loewy length $3$ contain the diamonds $$D=\Span_\KK\{a^ib^jxy, a^{i+1}b^jy, a^ib^{j+1}x,a^{i+1}b^{j+1}\}:$$

$$
\begin{tikzpicture}[->]

 \node(1) at (0,0) {$\cdot$};
\node(2) at (1,0) {$\cdot$};
\node(3) at (0,1) {$\cdot$};
\node(4) at (1,1)  {$\cdot$};

 \draw(1)--(3);
\draw(1)--(2);
\draw(2)--(4);
\draw(3)--(4);
\end{tikzpicture}
$$

To see that these are all indecomposable comodules, notice that a simple $S$ can extend with a string $C$ if and only if $S\cong \Top D$, $C\cong \Rad D$ or $ S\cong\Soc D$, $C\cong D/\Soc D$ for some diamond $D$.  Assume $M$ is an indecomposable comodule of Loewy length $3$. Since $M$ is an extension of $\Soc M$, which is a direct sum of simplies, and $M/\Soc M$, which has Loewy length $2$, hence a direct sum of strings. Therefore $M$ must be a diamond.

Since in the category of comodules ${}^{B^{m,n}_N(\lambda,s,t,k)}\Mm$, there are finitely many simples, strings and diamonds. So $B^{m,n}_N(\lambda,s,t,k)$ is finite representation type.
 \end{proof}

Our next task  is to classify the Hopf algebras $B^{m,n}(\lambda,s,t,k)$ up to isomorphism.

\begin{lemma} \label{Bmn iso}
Hopf algebras  $B^{m,n}(\lambda,s,t,k)\cong B^{m',n'}(\lambda',s',t',k')$ if and only if  for some $\alpha, \beta\in \KK^\times$, one of the following is satisfied:
\begin{enumerate}
\item $(m,n)=\pm(m',n')$, $(\lambda,s,t,k)=(\lambda',\alpha^2s',\beta^2t',\alpha\beta k')$
\item $(m,n)=\pm(n',m')$, $\lambda=\lambda'=1$, $(s,t,k)=(\alpha^2t', \beta^2s', \alpha\beta k')$.
\end{enumerate}
\end{lemma}

\begin{proof}
Denote by $a',b',x',y'$ the corresponding elements in $B^{m',n'}(\lambda',s',t',k')$ as in Definition \ref{defn Bmn}.
 Suppose there is a Hopf algebra (bialgebra) isomorphism $\varphi: B^{m,n}(\lambda,s,t,k)\to B^{m',n'}(\lambda',s',t',k')$. Then $\varphi$ preserves group-like elements and $\varphi(1)=1'$.

Since $\Delta'(\varphi(x))=\varphi\otimes\varphi(\Delta(x))=1'\otimes\varphi(x)+\varphi(x)\otimes \varphi(a)$, $\varphi(x)\in P(1',\varphi(a))$. On the other hand, $\varphi(x)$ is not a trivial skew primitive element since $\varphi$ is invertible. Hence either $\varphi(a)=a'$ or $\varphi(a)=b'$. We discuss these two cases separately:

{\bf Case 1:}  When $\varphi(a)=a'$, $\varphi(b)=b'$. Hence $a'^m=b'^n$. Therefore $(m,n)=l(m',n')$ for some integer $l$. Since $\varphi$ is invertible, $(m',n')=l'(m,n)$ for some integer $l'$. So, $l=\pm1$ or $(m,n)=(m',n')=(0,0)$.

Next assume $\varphi(x)=\alpha x'+\gamma (1'-a')$.

On one hand, $\varphi(a^{-1}xa)=\varphi(-x)=-\alpha x'-\gamma (1'-a')$. On the other hand, $\varphi(a^{-1}xa)= {a'}^{-1}\varphi(x)a'=-\alpha x'+\gamma(1'-a')$. It follows that $\gamma (1'-a')=0$. (char $\KK\neq 2$)

One can easily show that $\varphi(b)=b'$ and $\varphi(y)=\beta y'$ similarly.
Then it follows immediately that $s=\alpha^2s'$, $t=\beta^2 t'$.

Last, since $k(1-a'b')= \varphi(xy+\lambda yx)= \alpha\beta (x'y'+\lambda y'x')$, if $\lambda\neq \lambda'$, then $x'y'\in B^{m,n}(\lambda',s',t',k')_0$ which is a contradiction. Hence $\lambda=\lambda'$ and then $k=\alpha\beta k'$.

{\bf Case 2:}  When $\varphi(a)=b'$, $\varphi(b)=a'$.  Hence $b'^m=a'^n$. Therefore $(n,m)=k(m',n')$ for some integer $k$. Similarly since $\varphi$ is invertible $k=\pm 1$ or $(m,n)=(m',n')=(0,0)$.

 Next assume $\varphi(x)=\alpha y'+\gamma (1'-b')$.

On one hand, $\varphi(a^{-1}xa)=\varphi(-x)=-\alpha y'-\gamma (1'-b')$. On the other hand, $\varphi(a^{-1}xa)= {a'}^{-1}\varphi(x)a'=-\frac{\alpha}{\lambda} y'+\gamma(1'-b')$. It follows that $\lambda=1$ and $\gamma(1'-b')=0$. (char $\KK\neq 2$)

One can easily show that $\varphi(b)=a'$ and $\varphi(y)=\beta x'$ similarly.
Then it follows immediately that $s=\alpha^2t'$, $t=\beta^2 s'$ and $k=\alpha\beta k'$.
\end{proof}

\begin{remark}\label{def iso}
Notice that $B^{m,n}=B^{-m,-n}$. For $\alpha,\beta\in\KK^\times$, denote by $\varphi_{\alpha, \beta}: B^{m,n}(\lambda,s,t,k)\to B^{m,n}(\lambda',s',t',k')$ the isomorphism
$$\varphi_{\alpha, \beta}(a)=a', \varphi_{\alpha, \beta}(b)=b', \varphi_{\alpha, \beta}(x)=\alpha x', \varphi_{\alpha, \beta}(y)=\beta y'$$
and  $\psi_{\alpha, \beta}:B^{m,n}(\lambda,s,t,k)\to B^{n,m}(\lambda',s',t',k')$ the isomorphism
$$\psi_{\alpha, \beta}(a)=b', \psi_{\alpha, \beta}(b)=a', \psi_{\alpha, \beta}(x)=\alpha y', \psi_{\alpha, \beta}(y)=\beta x'$$
\end{remark}

%\begin{corollary}
%With the notation in Remark \ref{def iso} and denote $B=B^{m,n}(\lambda,s,t,k)$. Then
%\begin{enumerate}
%\item $\varphi_{\alpha,\beta}\in\Aut(B)$ if and only if $(\alpha^2-1)s=0$,  $(\beta^2-1)t=0$ and $(\alpha\beta-1)k=0$.
%\item When $\lambda=1$, $\psi_{\alpha,\beta}\in\Aut(B)$ if and only if $\alpha^2t=s$,  $\beta^2s=t$ and $(\alpha\beta-1)k=0$.
%\end{enumerate}
%\end{corollary}
%\begin{proof}
%Apply $\varphi_{\alpha,\beta}$ (respectively $\psi_{\alpha,\beta}$) on both sides of the relations in Definition \ref{defn Bmn}.
%\end{proof}

In fact, there are more constraint on the parameters $\lambda,s,t,k$.

\begin{lemma}\label{constraints}
For the algebra $B^{m,n}(\lambda,s,t,k)$, one of the following happens:\\
(1) $\lambda=1$ \\
(2) $\lambda=-1$ and $k=0$\\
(3) $k=s=t=0$.
\end{lemma}
\begin{proof}
Multiply $x$ on the left and right of the equation $xy+\lambda yx=k(1-ab)$, one have:
$$
xyx+\lambda yx^2=k(1-ab)x,  x^2y+\lambda xyx=kx(1-ab)
$$
So
$$
xyx=-\lambda yx^2+k(1-ab)x=-\frac{1}{\lambda}x^2y+\frac{k}{\lambda}x(1-ab).
$$
Plug in $x^2=s(1-a^2)$ and simplify the second equation, one have:
$$
\lambda kx-\lambda^2 sy=kx-sy
$$
That is $(\lambda-1)[kx+(\lambda+1)sy]=0$.

So either $\lambda=\pm 1$ or $k=s=0$ ($char \KK\neq 2$).  Also if $\lambda=-1$, $k=0$.

Similarly by multiplying $y$ on both sides, one have either $\lambda=\pm 1$ or $k=t=0$.
\end{proof}

To summarize, combining Lemma \ref{Bmn iso} and Lemma \ref{constraints} we have the following classification:

\begin{theorem}\label{classify B}
Let $\KK$ be an algebraic closed field with $char \KK\neq 2$. Let $H$ be a discrete corepresentation type pointed Hopf algebra over $\KK$ and  $B$ is the maximal connected Hopf subalgebra of $H$. If there are two different group-likes $a$,$b$ such that $\dim_\KK P(1,a)>1, \dim_\KK P(1,b)>1$. Then $B$ has an $\Ext$-quiver $Q_B=Q^{m,n}$ for some integers $m+n\in 2\mathbb Z$, such that $m\neq n$ or $m=n=0$ and $B$ is isomorphic to exactly one of the following:
\begin{enumerate}
\item $B^{m,n}(\lambda, 0,0,0)$, $\lambda\in\KK^\times$ and $\lambda^{gcd(m,n)}=1$ if $(m,n)\neq (0,0)$;
\item $B^{m,n}(-1,0,1,0)$, if both $m,n$ are even;
\item $B^{m,n}(-1,1,0,0)$, if both $m,n$ are even;
\item $B^{m,n}(-1,1,1,0)$, if both $m,n$ are even;
\item $B^{m,n}(1,1,1,k)$, $k\in\KK$;
\item $B^{m,n}(1,1,0,0)$;
\item[(6')] $B^{m,n}(1,0,1,0)$;
\item $B^{m,n}(1,1,0,1)$;
\item[(7')]  $B^{m,n}(1,0,1,0)$;
\item $B^{m,n}(1,0,0,1)$.
\end{enumerate}
Notice that $(6)$ and $(6')$; $(7)$ and $(7')$ are isomorphic if and only if $m=-n$.

For a Hopf algebra $B$ in each isomorphism class above,  the  automorphism group $\Aut(B)$ is listed in the table below:

(I) $m+n\neq 0$

\begin{tabular}{|c|c|c|c|}
\hline &&& \\
&$(m,n)$ & $(\lambda,s,t,k)$ & $\Aut(B)$\\
\hline 
(1)& & $(\lambda, 0,0,0)$ & $\{\varphi_{\alpha,\beta}: \alpha, \beta\in \KK^\times\}\cong \KK^\times\times \KK^\times$ \\
\hline  
(2)& $m,n\in 2\mathbb Z$ & $(-1,0,1,0) $ &  $\{\varphi_{\alpha,\beta}: \alpha\in \KK^\times, \beta^2=1\}\cong \KK^\times\times \mathbb Z/2\mathbb Z$ \\
\hline  
(3)& $m,n\in 2\mathbb Z$ & $(-1,1,0,0) $ &  $\{\varphi_{\alpha,\beta}: \alpha^2=1, \beta\in \KK^\times\}\cong\KK^\times\times \mathbb Z/2\mathbb Z$\\
\hline  
(4)& $m,n\in 2\mathbb Z$ & $(-1,1,1,0) $ &  $\{\varphi_{\alpha,\beta}: \alpha^2=1=\beta^2\}\cong\mathbb Z/2\mathbb Z\times \mathbb Z/2\mathbb Z$\\
\hline 
(5A)& & $(1,1,1,0) $ & $\{\varphi_{\alpha,\beta}: \alpha^2=1=\beta^2\}\cong\mathbb Z/2\mathbb Z\times \mathbb Z/2\mathbb Z$\\
\hline  
(5B)& & $(1,1,1,k\neq 0) $ & $\{\varphi_{\alpha,\beta}: \alpha^2=\beta^2=\alpha\beta=1\}\cong \mathbb Z/2\mathbb Z$ \\
\hline 
(6)& & $(1,1,0,0) $ & $\{\varphi_{\alpha,\beta}: \alpha^2=1, \beta\in \KK^\times\}\cong\KK^\times\times \mathbb Z/2\mathbb Z$\\
\hline
(6')& & $(1,0,1,0) $ & $\{\varphi_{\alpha,\beta}: \beta^2=1, \beta\in \KK^\times\}\cong\KK^\times\times \mathbb Z/2\mathbb Z$\\
\hline
(7)& & $(1,1,0,1) $ & $\{\varphi_{\alpha,\beta}: \alpha^2=1=\alpha\beta\}\cong\mathbb Z/2\mathbb Z$\\
\hline
(7')& & $(1,0,1,1) $ & $\{\varphi_{\alpha,\beta}: \beta^2=1=\alpha\beta\}\cong\mathbb Z/2\mathbb Z$\\
\hline
(8)& & $(1,0,0,1) $ & $\{\varphi_{\alpha,\beta}: \alpha\beta=1\}\cong\KK^\times$\\
\hline
\end{tabular}

%\begin{tabular}{|c|c|c|c|c|c|}
%\hline &&&&& \\
%&$(m,n)$ & $(\lambda,s,t,k)$ & $\varphi_{\alpha, \beta}$ & $\psi_{\alpha, \beta}$ & $\Aut(B)$\\
%\hline &&&&&  \\
%(1)& & $(\lambda, 0,0,0)$ & $\alpha, \beta\in \KK^\times$ & $\emptyset$ & $\KK^\times\times \KK^\times$ \\
%\hline &&&&&  \\
%(2)& $m,n\in 2\mathbb Z$ & $(-1,0,1,0) $ &  $\alpha\in \KK^\times, \beta^2=1$ & $\emptyset$ & $\KK^\times\times \mathbb Z/2\mathbb Z$ \\
%\hline
%(3)& $m,n\in 2\mathbb Z$ & $(-1,1,0,0) $ &  $\alpha^2=1, \beta\in \KK^\times$ & $\emptyset$ & $\KK^\times\times \mathbb Z/2\mathbb Z$\\
%\hline
%(4)& $m,n\in 2\mathbb Z$ & $(-1,1,1,0) $ &  $\alpha^2=1=\beta^2$ & $\emptyset$ & $\mathbb Z/2\mathbb Z\times \mathbb Z/2\mathbb Z$\\
%\hline
%(5A)& & $(1,1,1,0) $ & $\alpha^2=1=\beta^2$ & $\emptyset$ & $\mathbb Z/2\mathbb Z\times \mathbb Z/2\mathbb Z$\\
%\hline
%(5B)& & $(1,1,1,k\neq 0) $ & $\alpha^2=\beta^2=\alpha\beta=1$ & $\emptyset$ & $\mathbb Z/2\mathbb Z$ \\
%\hline
%(6)& & $(1,1,0,0) $ & $\alpha^2=1, \beta\in \KK^\times$ & $\emptyset$ & $\KK^\times\times \mathbb Z/2\mathbb Z$\\
%\hline
%(6')& & $(1,0,1,0) $ & $\beta^2=1, \beta\in \KK^\times$ & $\emptyset$ & $\KK^\times\times \mathbb Z/2\mathbb Z$\\
%\hline
%(7)& & $(1,1,0,1) $ & $\alpha^2=1=\alpha\beta$ & $\emptyset$ & $\mathbb Z/2\mathbb Z$\\
%\hline
%(7')& & $(1,0,1,1) $ & $\beta^2=1=\alpha\beta$ & $\emptyset$ & $\mathbb Z/2\mathbb Z$\\
%\hline
%(8)& & $(1,0,0,1) $ & $\alpha\beta=1$ & $\emptyset$ & $\KK^\times$\\
%\hline
%\end{tabular}
\vskip 5pt

(II) $m+n=0$

\begin{tabular}{|c|c|c|c|}
\hline &&& \\
&$(m,n)$ & $(\lambda,s,t,k)$  & $\Aut(B)$\\
\hline &&&  \\
(1A)&$m,n\neq 0$ & $(1, 0,0,0)$ & $\{\varphi_{\alpha,\beta}, \psi_{\alpha, \beta}: \alpha, \beta\in \KK^\times\}\cong \small{(\KK^\times\times\KK^\times)\rtimes \mathbb Z/2\mathbb Z }$\\
\hline
(1B) &$m,n\neq 0$ & $(\lambda\neq 1, 0,0,0)$ & $\{\varphi_{\alpha,\beta}: \alpha, \beta\in \KK^\times\}\cong\KK^\times\times \KK^\times$ \\
\hline 
(2)& $m,n\in 2\mathbb Z$ & $(-1,0,1,0) $ &  $\{\varphi_{\alpha,\beta}: \alpha\in \KK^\times, \beta^2=1\}\cong\KK^\times\times \mathbb Z/2\mathbb Z$ \\
\hline 
(3)& $m,n\in 2\mathbb Z$ & $(-1,1,0,0) $ &  $\{\varphi_{\alpha,\beta}: \alpha^2=1, \beta\in \KK^\times\}\cong\KK^\times\times \mathbb Z/2\mathbb Z$\\
\hline 
(4)& $m,n\in 2\mathbb Z$ & $(-1,1,1,0) $ &  $\{\varphi_{\alpha,\beta}: \alpha^2=1=\beta^2\}\cong\mathbb Z/2\mathbb Z\times \mathbb Z/2\mathbb Z$\\
\hline  
(5A)& & $(1,1,1,0) $ & $\{\varphi_{\alpha,\beta}, \psi_{\alpha, \beta}: \alpha^2=1=\beta^2\}\cong D_4$\\
\hline 
(5B)& & $(1,1,1,k\neq 0) $ & $\{\varphi_{\alpha,\beta}, \psi_{\alpha, \beta}: \alpha^2=\beta^2=\alpha\beta=1\}\cong\mathbb Z/2\mathbb Z\times \mathbb Z/2\mathbb Z$\\
\hline  
(6)& & $(1,1,0,0) $ & $\{\varphi_{\alpha,\beta}: \alpha^2=1, \beta\in \KK^\times\}\cong\KK^\times\times \mathbb Z/2\mathbb Z$\\
\hline  
(7)& & $(1,1,0,1) $ & $\{\varphi_{\alpha,\beta}: \alpha^2=1=\alpha\beta\}\cong \mathbb Z/2\mathbb Z$\\
\hline  
(8)& & $(1,0,0,1) $ & $\{\varphi_{\alpha,\beta}, \psi_{\alpha, \beta}: \alpha\beta=1\}\cong Dih(\KK^\times)^*$\\
\hline
\multicolumn{4}{l}{\small *$Dih(G)$ is the generalized dihedral group of an abelian group $G$.}  \\
    \end{tabular}
%\end{tabular}

\end{theorem}
%\begin{tabular}{|c|c|c|c|c|c|}
%\hline &&&&& \\
%&$(m,n)$ & $(\lambda,s,t,k)$ & $\varphi_{\alpha, \beta}$ & $\psi_{\alpha, \beta}$ & $\Aut(B)$\\
%\hline &&&&&  \\
%(1A)&$m,n\neq 0$ & $(1, 0,0,0)$ & $\alpha, \beta\in \KK^\times$ & $\alpha, \beta\in \KK^\times$ & \small{$(\KK^\times\times\KK^\times)\rtimes \mathbb Z/2\mathbb Z $}\\
%\hline &&&&&  \\
%(1B) &$m,n\neq 0$ & $(\lambda\neq 1, 0,0,0)$ & $\alpha, \beta\in \KK^\times$ & $\emptyset$ &$\KK^\times\times \KK^\times$ \\
%\hline &&&&&  \\
%(2)& $m,n\in 2\mathbb Z$ & $(-1,0,1,0) $ &  $\alpha\in \KK^\times, \beta^2=1$ & $\emptyset$ & $\KK^\times\times \mathbb Z/2\mathbb Z$ \\
%\hline &&&&& \\
%(3)& $m,n\in 2\mathbb Z$ & $(-1,1,0,0) $ &  $\alpha^2=1, \beta\in \KK^\times$ & $\emptyset$ & $\KK^\times\times \mathbb Z/2\mathbb Z$\\
%\hline &&&&& \\
%(4)& $m,n\in 2\mathbb Z$ & $(-1,1,1,0) $ &  $\alpha^2=1=\beta^2$ & $\emptyset$ & $\mathbb Z/2\mathbb Z\times \mathbb Z/2\mathbb Z$\\
%\hline &&&&& \\
%(5A)& & $(1,1,1,0) $ & $\alpha^2=1=\beta^2$ & $\alpha^2=1=\beta^2$ & $D_4$\\
%\hline &&&&& \\
%(5B)& & $(1,1,1,k\neq 0) $ & $\alpha^2=\beta^2=\alpha\beta=1$ & $\alpha^2=\beta^2=\alpha\beta=1$ & $\mathbb Z/2\mathbb Z\times \mathbb Z/2\mathbb Z$\\
%\hline &&&&& \\
%(6)& & $(1,1,0,0) $ & $\alpha^2=1, \beta\in \KK^\times$ & $\emptyset$ & $\KK^\times\times \mathbb Z/2\mathbb Z$\\
%\hline &&&&& \\
%(7)& & $(1,1,0,1) $ & $\alpha^2=1=\alpha\beta$ & $\emptyset$ & $\mathbb Z/2\mathbb Z$\\
%\hline &&&&& \\
%(8)& & $(1,0,0,1) $ & $\alpha\beta=1$ & $\alpha\beta=1$ & $Dih(\KK^\times)$\\
%\hline
%\end{tabular}

\vskip10pt
Lastly we make some remarks about the structure for the whole algebra $H$. Recall that by Theorem \ref{B tensor G} (3), any $g\in G(H)$ acts on $B$ by conjugation. That is, $\ad(g)(-)=g(-)g^{-1}$ defines a group homomorphism $\ad: G\to \Aut(B)$.  Denote by $\Gamma_0:=\langle\varphi_{-1,-\lambda}, \varphi_{-\frac{1}{\lambda},-1}\rangle\leq\Aut(B):=\Gamma$.

\begin{lemma}\label{index 2}
\begin{enumerate}
\item The map $\ad: G\to \Aut(B)$ yields a two dimensional representation of the group $G(H)$.
\item $[G: C_{G(H)}(G(B))]\leq 2$ and hence the centralizer $C_{G(H)}(G(B))$ is a normal subgroup of $G(H)$.
\item $h\in C_{G(H)}(G(B))$ if and only if $\ad(h)=\varphi_{\alpha,\beta}$.
\item If $\ad(h)=\varphi_{\alpha,\beta}$ and $\ad(g)=\psi_{\gamma,\delta}$ , then $\ad (g^{-1}hg)=\varphi_{\beta,\alpha}$.
\item The restriction $\ad:G(B)\to \Gamma_0$ is a group homomorphism.
\item $[\Gamma:C_\Gamma(\Gamma_0)]\leq 2$.
\end{enumerate}
\end{lemma}
\begin{proof}
(1) Let $\rho: G(H)\to GL_2(\KK)$ be the group homomorphism such that
 $$\rho(g)=\begin{cases} \begin{bmatrix}\alpha&0\\0&\beta\end{bmatrix}, & \ad(g)=\varphi_{\alpha,\beta}\\
  \begin{bmatrix}0&\alpha\\\beta&0\end{bmatrix}, & \ad(g)=\psi_{\alpha,\beta}\end{cases}$$

  (2) Denote by $C=C_{G(H)}(G(B))$. $g\in C$ if and only if $\ad g=\varphi_{\alpha,\beta}$. If $g,h\not\in C$, then $\ad g=\psi_{\alpha,\beta}$, and $\ad h=\psi_{\gamma,\delta}$. Since $\ad (gh)= \varphi_{\alpha\delta,\beta\gamma}$,  $gh\in C$. So cosets $gC=hC$ for all $g,h\not\in C$. Therefore either $G(H)=C$ or $G(H)=C\cup gC$ for some $g\not\in C$.

  (3) and (4) is straightforward.
\end{proof} 

It is well-known that a pointed Hopf algebra $H$ can be recovered from its link indecomposable component $B$ as a crossed product $H = B\#_{\sigma} k(G/G(B))$, for a suitable group $G$ containing $G(B)$ as a normal subgroup (see Theorem \ref{B tensor G}). The construction requires an action of $G$ on $B$ as Hopf algebra automorphisms, in other words a homomorphism $G \rightarrow \operatorname{Aut}(B)$. In light of the above results on $G(B)$ and $\operatorname{Aut}(B)$, it is natural to ask whether it is possible to classify the crossed products which can be built out of an ordered pair $(B,G)$. Unfortunately this question appears to lie outside the scope of this article, so instead we pose it as a problem to be discussed in a sequel: 
 
\begin{problem} 
Let $B$ be one of the pointed, corepresentation discrete Hopf algebras in Theorem \ref{classify B}. Let $G$ be a group containing $G(B)$ as a normal subgroup. Classify (up to Hopf algebra isomorphism) the pointed Hopf algebras $H = B\#_{\sigma}k(G/G(B))$ with link indecomposable component $B$ and grouplike elements $G(H) = G$. 
\end{problem}

\section{Appendix}
 
In this section, we  provide a full list of  pointed Hopf algebras of discrete corepresentation type over an algebraically closed field $\KK$ with $char \KK=0$. 

(I) $H\cong \KK G$ for some group $G$.  

 (II)  
 $H\cong H_n(d, G, g, \chi, \mu)$, where integers $n\geq 2$ and $d|n$,  $G$ is a group and $g\in Z(G)$ is an element of order $n$, $\chi$ is a 1-dimensional character such that $\chi(g)=q$ is a primitive $d$-th root of unity and $\mu\in \KK$ is a scalar which may be non-zero only if $\chi^d=1$.

 The Hopf algebra $H_n(d, G, g, \chi, \mu)$ is generated by $h\in G$ and $x$ subject to relations:
 $$
 xh=\chi(h)hx, x^d=\mu(g^d-1).
 $$
The coalgebra structure is given by:
$$
\Delta (h)=h\otimes h, \Delta (x)=1\otimes x+x\otimes g, \epsilon (h)=1, \epsilon(x)=0.
$$
The antipode is $S(h)=h^{-1}, S(x)=-xg^{-1}$.

(II') $H\cong H_1(G,\chi)$, where $G$ is a group and $\chi$ is a 1-dimensional character.

   The Hopf algebra $H_1(G,\chi)$ is generated by group-likes $h\in G$ and a primitive element $x$ subject to relations:
 $$
 xh=\chi(h)hx.
 $$
 The coalgebra structure is given by:
$$
\Delta (h)=h\otimes h, \Delta (x)=1\otimes x+x\otimes 1, \epsilon (h)=1, \epsilon(x)=0.
$$
The antipode is $S(h)=h^{-1}, S(x)=-x$.

(III)  $H\cong H_\infty(d, G, g, \chi, \mu)$, where $G$ is a group and $g\in Z(G)$ is an element of infinite order, $\chi$ is a 1-dimensional character such that $\chi(g)=q$ is a primitive $d$-th root of unity and $\mu\in \KK$ is a scalar which may be non-zero only if $\chi^d=1$.

 The Hopf algebra $H_\infty(d, G, g, \chi, \mu)$ is generated by $h\in G$ and $x$ subject to relations:
 $$
 xh=\chi(h)hx, x^d=\mu(g^d-1).
 $$
The coalgebra structure is given by:
$$
\Delta (h)=h\otimes h, \Delta (x)=1\otimes x+x\otimes g, \epsilon (h)=1, \epsilon(x)=0.
$$
The antipode is $S(h)=h^{-1}, S(x)=-xg^{-1}$.

(III') $H\cong H_\infty(G, g,\chi,\lambda)$, where $G$ is a group, $g\in Z(G)$ has infinite order, $\chi$ is a 1-dimensional character such that $\chi(g)=1$ or $\chi(g)$ is not a root of unity and $\lambda\in (\KK G)^o$ is an element in the finite dual Hopf algebra such that $\lambda(hf)=\chi(h)\lambda(f)+\lambda(h)$.

   The Hopf algebra $H_\infty(G, g, \chi,\lambda)$ is generated by group-likes $h\in G$ and a $1$-$g$ skew primitive $x$ subject to relations:
 $$
 xh=\chi(h)hx+\lambda(h)(h-gh).
 $$

The coalgebra structure is given by:
$$
\Delta (h)=h\otimes h, \Delta (x)=1\otimes x+x\otimes g, \epsilon (h)=1, \epsilon(x)=0.
$$
The antipode is $S(h)=h^{-1}, S(x)=-xg^{-1}$.

(IV) The Hopf algebras described in the following Theorem.

 \begin{theorem}\label{H main}
Let $\KK$ be an algebraic closed field with $char \KK=0$. Let $H$ be a discrete corepresentation type pointed Hopf algebra over $\KK$ with L.I.C. $B$. Then either $H$ is coseial or $B\cong B^{m,n}(\lambda, k,s,t)$ is isomorphic to one of the pointed Hopf algebras listed in Theorem \ref{classify B}, and there is a group $G$ together with a group homomorphism $G\to \Aut B$, a cocycle $\sigma: G\times G\to G(B)$ and Hopf algebra isomorphism $H\cong B\#_\sigma \KK G$.
\end{theorem}

\bigskip\bigskip\bigskip

\begin{center}
{\sc Acknowledgment}
\end{center}

MI was supported  by a Simons Collaboration Grant. SZ is supported by NSF of China No. 12201321.

%\end{center}

%\newpage
%\bigskip\bigskip

\vspace*{3mm}
\begin{flushright}
\begin{minipage}{148mm}\sc\footnotesize

Miodrag Cristian Iovanov,
University of Iowa
Department of Mathematics, McLean Hall,
Iowa City, IA, USA  %\\
%Miodrag Cristian Iovanov\\
%University of Southern California\\
%Department of Mathematics, 3620 South Vermont Ave. KAP 400C \\
%Los Angeles, California 90089-2532\\
and 
Simion Stoilow Institute of the Romanian Academy, 
PO-Box 1-764, 014700 Bucharest, Romania, 
%University of Bucharest, Faculty of Mathematics, Str. Academiei 14\\
%RO-010014, Bucharest, Romania\\
%State University of New York (Buffalo)\\
%Department of Mathematics, 244 Mathematics Building\\
%Buffalo, NY 14260-2900, USA\\
{\it E--mail address}: {\tt
yovanov@gmail.com; miodrag-iovanov@uiowa.edu}\vspace*{3mm}

\vskip 5pt
Emre Sen, 
University of Iowa, 
Department of Mathematics, McLean Hall,
Iowa City, IA, USA,
{\it E--mail address}: {\tt emresen641@gmail.com}\vspace*{3mm}

\vskip 5pt
Alexander Sistko,
Manhattan College,
Department of Mathematics,
Riverdale, NY, USA,
{\it E--mail address}: {\tt asistko01@manhattan.edu}\vspace*{3mm}

\vskip 5pt

Shijie Zhu,
Nantong University,
School of Science,
Nantong, Jiangsu, China,
{\it E--mail address}: {\tt shijiezhu0011@gmail.com; shijiezhu@ntu.edu.cn}\vspace*{3mm}

\end{minipage}
\end{flushright}
\end{document}